\documentclass[psamsfonts, reqno]{amsart}
\newcommand\abs[1]{\left|#1\right|}
\usepackage[margin=1in]{geometry}
\usepackage{tikz-cd}

\usepackage{amssymb,amsfonts,amsmath}
\usepackage[all,arc]{xy}
\usepackage{enumerate}
\usepackage{mathrsfs}
\usepackage{mathtools}
\usepackage{thmtools}
\usepackage{enumitem}
\usepackage{bm}
\usepackage{graphicx}
\usepackage{hyperref}
\usepackage{url}
\usepackage{xcolor}

\newtheorem{thm}{Theorem}[section]

\newtheorem{prop}[thm]{Proposition}
\newtheorem*{prop*}{Proposition}
\newtheorem{lem}[thm]{Lemma}
\newtheorem{conj}[thm]{Conjecture}

\newtheorem*{thm*}{Theorem}
\newtheorem*{lem*}{Lemma}

\theoremstyle{definition}
\newtheorem{defn}[thm]{Definition}

\newtheorem{notn}[thm]{Notation}

\newtheorem*{defn*}{Definition}
\newtheorem*{fact*}{Fact}

\theoremstyle{remark}
\newtheorem{rem}[thm]{Remark}

\newtheorem*{rem*}{Remark}
\newtheorem*{rems*}{Remarks}

\declaretheoremstyle[notefont=\bfseries,notebraces={}{},%
headpunct={},postheadspace=1em]{mystyle}

\newcommand{\Q}{\mathbf{Q}}
\newcommand{\Z}{\mathbf{Z}}

\newcommand{\C}{\mathbf{C}}

\DeclareMathOperator{\res}{res}

\DeclareMathOperator{\Hom}{Hom}
\DeclareMathOperator{\tr}{tr}
\DeclareMathOperator{\Mod}{mod}

\DeclareMathOperator{\Br}{Br}
\DeclareMathOperator{\Gal}{Gal}
\DeclareMathOperator{\ord}{ord}
\DeclareMathOperator{\LC}{LC}

\newcommand{\SL}[2]{\mathrm{SL}_{#1}{#2}}
\newcommand{\HilbertSymbol}[3]{\ensuremath{\left({\dfrac{#1,#2}{#3}}\right)}}

\newcommand\eqdef{\mathrel{\overset{\makebox[0pt]{\mbox{\normalfont\tiny\rmfamily def}}}{=}}}
\newcommand\mathspace{\mkern\medmuskip}

\renewcommand\Im{\operatorname{Im}}

\makeatletter
\let\@@pmod\pmod
\DeclareRobustCommand{\pmod}{\@ifstar\@pmods\@@pmod}
\def\@pmods#1{\mkern4mu({\operator@font mod}\mkern 6mu#1)}
\makeatother

\makeatletter
\let\c@equation\c@thm
\makeatother

\makeatletter
\newcommand*{\house}[1]{%
	\mathord{%
		\mathpalette\@house{#1}%
	}%
}
\newcommand*{\@house}[2]{%
	\dimen@=\fontdimen8 %
	\ifx#1\scriptscriptstyle\scriptscriptfont
	\else\ifx#1\scriptstyle\scriptfont
	\else\textfont\fi\fi
	3 %
	\sbox0{%
		$#1%
		\vrule width\dimen@\relax
		\overline{%
			\kern2\dimen@
			\begingroup 
			#2%
			\endgroup
			\kern2\dimen@
		}%
		\vrule width\dimen@\relax
		\mathsurround=1.5\dimen@ 
		$%
	}%
	\ht0=\dimexpr\ht0-\dimen@\relax
	\dp0=\dimexpr\dp0+2\dimen@\relax
	\vbox{%
		\kern\dimen@ 
		\copy0 %
	}%
}

\usepackage[backend=biber,style=numeric, doi=false, isbn=false, giveninits=true, url=false]{biblatex}
\renewbibmacro{in:}{}
\graphicspath{ {./images/} }

\title{Arithmetic of the canonical component of the knot $7_4$}
\address{\newline Department of Mathematics,\newline Rice University,\newline Houston, TX 77005, USA} 
\email{nicholas.rouse@rice.edu}

\author{}

\date{\today}
\begin{document}
	
	\author{Nicholas Rouse}
	\begin{abstract}
	We prove two arithmetic properties of Dehn surgery points on the canonical component of the $\SL{2}{\C}$-character variety of the knot $7_4$. The first is that the residue characteristics of the ramified places of the Dehn surgery points form an infinite set, providing evidence for a conjecture of Chinburg, Reid, and Stover. The second is that the Dehn surgery points have infinite order in the Mordell-Weil group of the elliptic curve obtained by a simple birational transformation of the canonical component into Weierstrass form.
	\end{abstract}
	\maketitle
	
\section{Introduction}
	
Let $\Gamma$ be a finitely generated group, and let $X(\Gamma)$ denote the $\SL{2}{\C}$-character variety of $\Gamma$ (see Section \ref{subsec:character varieties}). When $\Gamma$ is the fundamental group of a compact 3-manifold $M$, work of Thurston and Culler--Shalen established $X(\Gamma)$ as a powerful tool in the study of the geometry and topology of $M$. 
The focus of this paper is arithmetic and algebraic properties of a particular component $C$ (the canonical component, see Section \ref{subsec:character varieties}) of $X(\Gamma)$ when $\Gamma$ is the fundamental group of a particular hyperbolic knot complement ($7_4$ of the tables of \cite{RolfsenBook} and $(15,11)$ in two-bridge notation.).
This has already been studied for different reasons (\cite{ChuCharacterVarieties}).
There are two themes to this: the first is that, following \cite{CRS}, we are particularly interested in a canonically defined quaternion algebra, $A_k(C)$, which is defined over the function field of $C$, $k(C)$, and specializes at Dehn surgery points of $C$ to quaternion algebras defined over number fields. The second theme is to view $C$ as an elliptic curve and to consider the Mordell-Weil group of naturally occurring number field points on $C$.

In more detail, for the first part, for knots satisfying an arithmetic condition on their Alexander polynomials (condition $(\star)$, see Section \ref{subsec:dim1canonicalcomponents}), Chinburg, Reid, and Stover show in \cite{CRS} that there are only finitely many rational primes lying under any finite prime ramifying the specializations of this quaternion algebra. Let us write $S$ for this set of rational primes.  Let us define $S_{D} \subseteq S$ to be the set of rational primes $p$ such that there is a specialization to the character of a hyperbolic Dehn surgery such that the quaternion algebra is ramified at some prime lying above $p$. When condition $(\star)$ fails, it is shown in \cite[Theorem 1.1(3)]{CRS} (using work of Harari \cite{Harari94}) that $S$ is infinite. They furthermore state as a conjecture \cite[Conjecture 6.7]{CRS} that
\begin{conj}\label{conj:inf_bad}
	Let $K$ be a hyperbolic knot in $S^3$ that fails condition $(\star)$, then, in the notation above, $S=S_{D}$.
\end{conj}

As we note in Section \ref{subsec:SevenFourAzumaya}, $7_4$ fails condition $(\star)$. Our first main result is:

\begin{thm}\label{maintheorem} Let $K$ be the knot $7_4$ and $T$ be the set of rational primes $p$ such that there exists a place $\mathfrak{p}$ lying above $p$ of the trace field of some hyperbolic Dehn surgery $(d,0)$ at which the canonical quaternion algebra associated to that surgery is ramified. Then $T$ and hence $S_{D}$ are infinite.
\end{thm}
\par
We now turn our attention to the second result, which concerns the arithmetic of Dehn surgery points in the Mordell-Weil group. Thought of as a variety embedded in $\mathbf{P}^2({\C})$, the projective closure of the canonical component $C$ of $7_4$ has singular points, but it is birational to a curve of genus one. Together with a choice of basepoint, such a curve is an elliptic curve. Concretely, the canonical component, $C$, is cut out by $R^3-R^2Z^2+2R^2-1=0$ and is birational via the coordinate change $R=x, Z=y/x$ to $E$, the affine variety cut out by $y^2=x^3+2x^2-1$.  The latter equation is a nonsingular Weierstrass equation and hence determines an elliptic curve by taking the unique point at infinity to be the basepoint. We may then regard a ``Dehn surgery point" on the elliptic curve to be any point in the image of the birational map $C \dashrightarrow E$. \par 
A basic fact about elliptic curves is that their points can be made into an abelian group, and (over $\C$, say) the $n$-torsion points form a subgroup isomorphic to $\Z/n\Z \times \Z/n\Z$. In general it is a difficult problem to produce infinite order points in a particular number field on a given elliptic curve. One of the few ways uses the theory of Heegner points (see e.g. \cite{HeegnerPointsMSRIBook}) which allows one to construct infinite order $\Q$-rational points on certain elliptic curves. However, experimental evidence suggested that hyperbolic Dehn surgery points were never torsion points, and our second main result is that in fact every hyperbolic Dehn surgery point has infinite order.
\begin{thm}\label{maintheorem2} Let $E$ be the elliptic curve defined by $y^2=x^3+2x^2-1$. With the conventions of the above paragraph, every hyperbolic Dehn surgery point has infinite order in the Mordell-Weil group of $E$.
\end{thm}
In our setting characters of hyperbolic Dehn surgeries have infinite order, and they are rational points over their trace fields. The proof of \ref{maintheorem2} (see Section \ref{TorsionPointsSection}) combines an algebraic fact (Proposition \ref{prop:2-adic valuation}) with mostly topological results of Bass, Hatcher, and Thurston. The algebraic fact shows that nonpositive $2$-adic valuation of the $x$-coordinate of a point on $E$ obstructs being a torsion point. The topological results imply that no Dehn surgery point can have such a form.

\subsection{Outline}
The paper is organized as follows. We introduce some background material on canonical components, trace field, and quaternion algebras in Section \ref{quaternionalgebrassection}. Then we discuss their generalization, Azumaya algebras and how they figure into studying ramification in Section \ref{AzumayaSection}. We then give a proof of Theorem \ref{maintheorem} in Section \ref{maintheoremproofsection}. There are many lemmas used in the proof, and we delay their proofs until Section \ref{proofs} so that the proof of \ref{maintheorem} may be read more easily. Finally in Section \ref{TorsionPointsSection} we provide some background material on elliptic curves and prove Theorem \ref{maintheorem2}.
\subsection{Acknowledgments} The author wishes to thank his advisor, Alan Reid, for suggesting the problems in this paper as well as his support and guidance in both the mathematical and writing phases of this paper's preparation. The author would also like to acknowledge the anonymous referees for their helpful comments and suggestions, with special thanks to the one who pointed out Theorem \ref{thm:nonlacunary}, which simplified the original argument.

\section{The Canonical Component, Traces Fields, and Quaternion Algebras} \label{quaternionalgebrassection}
In this section we provide some background material about character varieties and quaternion algebras. We then establish the canonical component of $7_4$ and a tractable form of the canonical quaternion algebra.
\subsection{Character Varieties} \label{subsec:character varieties}  We begin by recalling that, for a finitely generated group $\Gamma$, the $\SL{2}{\C}$-representation variety of $\Gamma$ is $R(\Gamma) = \Hom(\Gamma, \SL{2}{\C})$. Given a generating set $\{\gamma_i\}$, we identify a representation $\rho : \Gamma \rightarrow \SL{2}{\C}$ with $(\rho(\gamma_1),\dots,\rho(\gamma_n)) \subset \SL{2}{\C}^n \subset \C^{4n}$. Given a different choice of generators, there is a canonical isomorphism between the two subsets of $\C^{4n}$ obtained this way. Fixing an element $\gamma \in \Gamma$, we may define a map $I_{\gamma}$ on $R(\Gamma)$ that associates to a representation $\rho$ the trace of $\rho$. That is, $I_{\gamma} : R(\Gamma) \rightarrow \C$ is defined by $I_{\gamma}(\rho) = \tr \rho(\gamma) = \chi_{\rho}(\gamma)$. This $I_{\gamma}$ is a regular function on the algebraic set $R(\Gamma)$, and the ring $T$ generated by all such $I_{\gamma}$ turns out to be finitely generated. This is \cite[Proposition 1.4.1]{CullerShalen}. Fixing a generating set $I_{\gamma_1},\dots,I_{\gamma_m}$ for $T$, define a map $t: R(\Gamma) \rightarrow \C^m$ by $t(\rho) = (I_{\gamma_1}(\rho),\dots,I_{\gamma_m}(\rho))$. Then define the $\SL{2}{\C}$-character variety of $\Gamma$ to be $t(R(\Gamma)) \subset \C^m$. This is a closed algebraic set, and different choices of generators for $T$ give isomorphic algebraic sets. When $\Gamma$ is the fundamental group of the complement of a hyperbolic knot $K$ in $S^3$, we define its canonical component to be the irreducible component of $X(\Gamma)$ containing the character of the discrete and faithful representation of $\pi_1(S^3 \backslash K)$. We refer the reader to \cite{CullerShalen} for more detail.\par
\subsection{Computation of the Character Variety} \label{subsec:charactervarietycomputation}
We start with the fundamental group of the complement of $7_4$ in $S^3$ and some notation for the canonical component.
\begin{notn} Let $K$ be the knot $7_4$ in $S^3$. Write $\Delta_K(t) = 4t^2-7t+4$ for its Alexander polynomial. We also write $\Gamma$ for the fundamental group of the complement of $K$ in $S^3$. We use the following presentation
\[
\Gamma = 
\pi_1(S^3 \backslash K) = \langle a,b \mathspace \vert \mathspace aw^2 = w^2b \rangle,
\]
where $w = ab^{-1}ab^{-1}a^{-1}ba^{-1}b$. 
For a representation $\rho : \Gamma \rightarrow \SL{2}{\C}$, we conjugate so that
\[
\begin{aligned}
\rho(a) &= \begin{pmatrix}
		x & 1 \\
		0 & 1/x
		\end{pmatrix} \\
		\rho(b) &= \begin{pmatrix}
		x & 0 \\
		r & 1/x
		\end{pmatrix}.
		\end{aligned}
		\]
	\end{notn}
	In defining an algebraic set, one should generally avoid expressions like $1/x$, but here we use it as a shorthand for $y$ where $xy=1$. Also, note that this presentation for the fundamental group comes from the two-bridge normal form for $7_4$, namely $(15,11)$.  These facts and the following are in \cite[Section 5]{ChuCharacterVarieties}.
	\begin{prop}\label{charactervariety} If we write
		\[
		Z = \chi_{\rho}(a) = \chi_{\rho}(b) = x + \frac{1}{x},
		\]
		and
		\[
		R = \chi_{\rho}(ab^{-1}) = \tr \begin{pmatrix}
		1-r & x \\
		-r/x & 1
		\end{pmatrix} = 2-r,
		\] 
	then the $\SL{2}{\C}$ character variety has canonical component given by the vanishing of $R^3-R^2 Z^2+2 R^2-1$. 
	\end{prop}
\subsection{Quaternion Algebras over Fields} We now recall some facts about quaternion algebras (see, e.g., \cite[Ch.2]{MR}). Recall that a \textbf{quaternion algebra} $A$ over a field $F$ of characteristic not equal to $2$ is a $4$-dimensional central simple algebra over $F$. More concretely, $A$ is a $4$-dimensional algebra over $F$ admitting an  $F$-basis $\{1,i,j,ij\}$ with $i^2=a$, $j^2=b$, and $ij=-ji$ where $a,b \in F^{*}$. One may efficiently encode this information with a \textbf{Hilbert symbol}, $\HilbertSymbol{a}{b}{F}.$ Note that any quaternion algebra is described by many Hilbert symbols. \par
Though we will have occasion to consider quaternion algebra over function fields, our real objective is to study the quaternion algebras that are associated to Dehn surgery points. These are quaternion algebras over number fields. In this situation, there is a powerful classification theorem that in some sense justifies investigating the ramification set of Theorem \ref{thm:1.2OfCRS} and Theorem \ref{maintheorem} in the first place. We begin with a quaternion algebra $A$ over a number field $L$. Given a place $\mathfrak{p}$ of $L$, one may form the completion $L_{\mathfrak{p}}$ and extend $A$ to a quaternion algebra $A_{\mathfrak{p}} = A \otimes_L L_{\mathfrak{p}}$. There are exactly two isomorphism classes of quaternion algebras over the local field $L_{\mathfrak{p}}$. If $A_{\mathfrak{p}}$ is isomorphic to $M_2(L_{\mathfrak{p}})$ then $A$ is said to \textbf{split} at $\mathfrak{p}$. Otherwise, $A_{\mathfrak{p}}$ is the unique division quaternion algebra over $L_{\mathfrak{p}}$ and $A$ is said to \textbf{ramify} at $\mathfrak{p}$.
We state the version of the classification theorem for quaternion algebras over number fields as it appears in \cite[Theorem 7.3.6]{MR}.
\begin{thm}\label{classificationtheorem} Let $A$ be a quaternion algebra over the number field $L$ and let $\mathrm{Ram}(A)$ denote the set of places at which $A$ is ramified. Then,
	\begin{enumerate}
		\item $\mathrm{Ram}(A)$ is finite of even cardinality.
		\item Let $A_1$ and $A_2$ be two quaternion algebras over $L$. Then $A_1 \equiv A_2$ if and only if $\mathrm{Ram}(A_1) = \mathrm{Ram}(A_2)$.
		\item Let $S$ be any finite set of even cardinality of finite and nonreal infinite places, then there exists a quaternion algebra $A$ over $L$ with $\mathrm{Ram}(A) = S$. 
	\end{enumerate}
\end{thm}
There is a relatively easy way to compute the ramification sets. We use the following description of the ramification of a quaternion algebra over a $\mathfrak{p}$-adic field which is sufficient for our purposes.
\begin{thm} [{{\cite[Theorem 2.6.6.(b)]{MR}}}] \label{MRRamification}
Let $L$ be a non-dyadic $\mathfrak{p}$-adic field, with ring of integers $\mathcal{O}$ and maximal ideal $\mathfrak{p}$. Let $A = \HilbertSymbol{a}{b}{L}$, where $a,b \in \mathcal{O}$. If $a \notin \mathfrak{p}$, $b \in \mathfrak{p}\backslash \mathfrak{p}^2$, then $A$ splits if and only if $a$ is a square modulo $\mathfrak{p}$.
\end{thm}
\subsection{Number Fields and Quaternion Algebras Associated to Subgroups of \texorpdfstring{$\SL{2}{\C}$}{SL2(C)}}We next turn to some background information about subgroups of $\SL{2}{\C}$. A subgroup $\Gamma$ of $\SL{2}{\C}$ is \textbf{non-elementary} if its image in $\mathrm{PSL}_2\C$ has no finite orbit in its action on $\mathbf{H}^3\cup\widehat{\C}$. Given an non-elementary subgroup $\Gamma$ of $\SL{2}{\C}$, we define its \textbf{trace field} by $k_{\Gamma} = \Q\left(\tr \gamma \mathspace | \mathspace \gamma \in \Gamma\right)$ and \textbf{quaternion algebra} by the $k_{\Gamma}$-span of elements of $\Gamma$. That is,
	\[
	A_{\Gamma} = \left\{\sum_{\text{finite}}\alpha_i \gamma_i \mathspace \big| \mathspace \alpha_i \in k_{\Gamma}, \gamma_i \in \Gamma \right\}.
	\]
	As shown in \cite[p.78]{MR}, we may write a Hilbert symbol for this quaternion algebra as
	\[
	\HilbertSymbol{\chi(g)^2-4}{\chi(g,h)-2}{k_{\Gamma}},
	\]
	where $g, h$ are noncommuting hyperbolic elements of $\Gamma$. 
	In fact, this pointwise construction extends to define a quaternion algebra over the function field of the curve. 

\begin{prop}[{\cite[Corollary 2.9]{CRS}}] \label{CRSFunctionFieldHilbertSymbol} Let $\Gamma$ be a finitely generated group, and $C$ an irreducible component of the character variety of $\Gamma$ defined over the number field $k$. Assume that $C$ contains the character of an irreducible representation, and let $g,h \in \Gamma$ be two elements such that there exists a representation $\rho$ with character $\chi_{\rho} \in C$ for which the restriction of $\rho$ to $\langle g, h \rangle$ is irreducible. Then the canonical quaternion algebra $A_{k(C)}$ is described by the Hilbert symbol
	\[
	\HilbertSymbol{I_g^2-4}{I_{[g,h]}-2}{k(C)}.	
	\] 
\end{prop}
For the remainder of the section, let us specialize to the case of $K = 7_4$, $\Gamma = \pi_1(S^3 \backslash K)$, and $C$ the canonical component of the $\SL{2}{\C}$-character variety of $\Gamma$. Recall that $C$ is cut out by $R^3-R^2 Z^2+2 R^2-1$. We now give an explicit Hilbert symbol for the canonical quaternion algebra associated to $\Gamma$.
\begin{lem} \label{functionfieldhilbertsymbol} The canonical quaternion algebra over $k(C)$ is given by
	\[
	\HilbertSymbol{Z^2-4}{R-2}{k(C)}.	
	\]
	If we use the coordinate $r = R-2$, then the Hilbert symbol is given by
	\[
	\HilbertSymbol{-r^3+4r^2-4r-1}{-r}{k(C)}.
	\]
\end{lem}
\begin{proof} If we let $a,b$ be the two generators for the knot group, they satisfy the hypotheses of Proposition \ref{CRSFunctionFieldHilbertSymbol}. We then know that our Hilbert symbol is given by
	\[
	\HilbertSymbol{I_a^2-4}{I_{[a,b]} - 2}{k(C)}.
	\]
For the first term, we have that $I_a^2 - 4 = Z^2 - 4$. \par 
Then from the description of the canonical component, we have $Z^2R^2 = R^3 + 2R^2 - 1$. Since multiplying by a square doesn't affect the Hilbert symbol, we can substitute $Z^2 - 4$ with $Z^2R^2 - 4R^2$. Then,
	\[
	\begin{aligned}
	Z^2R^2 - 4R^2 &= R^3 + 2R^2 - 1 - 4R^2 \\
	&= R^3 - 2R^2 - 1.
	\end{aligned}
	\]
From Proposition \ref{charactervariety}, we may substitute the relation $R = 2-r$ to obtain
	\[
	\begin{aligned}
	R^3 - 2R^2 - 1 &= (2-r)^3 - 2(2-r)^2 - 1 \\
		&= -r^3 + 4r^2 - 4r -1.
	\end{aligned}
	\]
For the second term, we use the trace relations (\cite[p.121]{MR}):

\[
\begin{aligned}
I_{ab} &=  I_{a}I_{b} - I_{ab^{-1}} \\
&= Z^2 - R,
\end{aligned}
\]
so
\[
\begin{aligned}
I_{[a,b]} &= I_a^2 + I_b^2 + I_{ab}^2 - I_aI_bI_{ab} - 2 \\
&= 2Z^2 +\left(Z^2-R\right)^2 - Z^2\left(Z^2-R\right) - 2 \\
&= 2Z^2 - Z^2R + R^2 - 2.
\end{aligned}
\]
We can multiply $\left(2Z^2 - Z^2R + R^2 - 2\right) - 2 = 2Z^2 - Z^2R + R^2 - 4$ through by $R^2$ and use the relation from the canonical component to obtain
\[
\begin{aligned}
2Z^2R^2 - Z^2R^3 + R^4 - 4R^2 &= 2\left( R^3 + 2R^2 - 1 \right) - R\left( R^3 + 2R^2 - 1 \right) + R^4 - 4R^2 \\
&= 2R^3 + 4R^2 - 2 - R^4 - 2R^3 + R + R^4 - 4R^2 \\
&= R - 2. \\
\end{aligned}
\]
\end{proof}

Given the description of the canonical quaternion algebra over the function field $k(C)$, one may also pass back to the pointwise-defined quaternion algebras by specializing the entires of the canonical quaternion algebra to points on the curve $C$. One must pay attention to the field over which these quaternion algebras are defined however. Fortunately, the trace field and the residue field (in the sense of algebraic geometry) coincide.
\begin{lem}[{[Lemma 2.5]{CRS}}] \label{CRS:2.5} Let $C$ be an irreducible affine or projective curve defined over $\Q$. Let $\tilde{C}$ be the smooth projective completion of the normalization of the reduction of $C$. For any $z \in \tilde{C} \backslash \mathcal{I}(\tilde{C})$, let $\chi_{\rho} \in C$ be the associated character, i.e., the image of $z$ on $C$ under the rational map $\tilde{C} \rightarrow C$. Then
	\[
	k(z) = \Q\left(\tr(\rho(\gamma)) \mathspace | \mathspace \gamma \in \Gamma\right)=k_{\rho}.
	\]
is the trace field of some (hence any) representation $\rho \in R(\Gamma)$ with character $\chi_{\rho}$.	
\end{lem}
\begin{rem} Until Section \ref{TorsionPointsSection}, we can ignore the the normalizations, reductions, and smooth projective closures of $C$ because the canonical component is a smooth affine curve. It is not smooth at infinity, but our primary object of interest, Dehn surgery points, lie on $C$.
\end{rem}
\section{Extending Quaternion Algebras Over Function Fields to Azumaya Algebras} \label{AzumayaSection}
In this section we describe some of the algebro-geometric considerations for our problem. In particular, we explore the problem of extending a quaternion algebra defined over the function field of a scheme to an element of the Brauer group of that scheme. We begin with general discussion of Brauer groups before specializing to curves, and eventually to the canonical component coming from $7_4$.
\subsection{Brauer Groups of Schemes} For any scheme $X$, one defines the Brauer group by $\Br X = H^2_{\mathrm{\acute{e}t}}(X, \mathbf{G}_m)$, where $\mathbf{G}_m$ is the multiplicative group scheme. We will have no need for the details of \'{e}tale cohomology, and the reader may think of $X$ as a variety in this paper. Given this definition, we have an injection $\Br X \hookrightarrow \Br k(X)$ and exact sequence that describes precisely which elements of $\Br k(X)$ are in the image of this injection. We present it as it appears in \cite[Theorem 6.8.3]{QPoints}, though the result itself is due to Grothendieck and Gabber. 
	\begin{thm} \label{BrauerExactSequence}
		Let $X$ be a regular integral Noetherian scheme. Let $X^{(1)}$ be the set of codimension $1$ points of $X$. Then the sequence
		\[
		0 \rightarrow \Br X \rightarrow \Br k(X) \xrightarrow{\text{res}} \bigoplus_{x \in X^{(1)}} H^1(k(x), \Q/\Z)
		\]
		is exact with the caveat that one must exclude the $p$-primary part of all the groups if $X$ is of dimension $\leq 1$ and some $k(x)$ is imperfect of characteristic $p$, or if $X$ is of dimension $\geq 2$ and some $k(x)$ is of characteristic $p$.
	\end{thm}
In the above theorem, $k(x)$ is the residue field at the point $x$, and $\mathrm{res}$ denotes the residue homomorphism into the Galois cohomology group $H^1(k(x), \Q/\Z) = H^1(\Gal(k(x)^{sep}/k(x)), \Q/\Z)$. Note that smooth varieties are regular schemes. We say that $A_{k(X)}$ ``extends" over a point $x \in X$ if the residue is trivial at $x$. This exact sequence says that $A_{k(X)}$ extends to an element of $\Br X$ if and only if it has trivial residue at every codimension $1$ point $x$ in $X$. Elements of $\Br X$ are called \textbf{Azumaya algebras}. Quaternion Azumaya algebras are Azumaya algebras that locally look like quaternion algebras. Elements of $\Br k(X)$ that do not belong to $\Br X$ are characterized in terms of their ramification sets as the following result shows.
\begin{thm}[{\cite[Theorem 1.1.(3)]{CRS}}] \label{thm:CRS-Harari} Let $\Gamma$ be a finitely generated group with $\SL{2}{\C}$ character variety $X(\Gamma)$. Let $C$ be a geometrically integral $1$-dimensional subvariety defined over $\Q$ that contains the character of an irreducible representation and write $\widetilde{C}$ for the smooth projective closure of the normalization of $C$. Finally suppose that $A_{k(C)}$ is not in the image of the canonical injection $\Br \widetilde{C} \rightarrow \Br k(C)$. Then there is no finite set of places $S$ of $\Q$ with the following property: the $k(w)$-quaternion algebra $A_{\rho} \otimes_{k_\rho}k(w)$ is unramified outside the places of $k(w)$ over $S$ for all but finitely many smooth points $w \in C(\overline{\Q})$ for which $\rho = \rho_{w}$ is absolutely irreducible.
\end{thm}
\begin{rem} Both the $\SL{2}{\C}$ character variety and canonical component for $7_4$ are defined over $\Q$, and the canonical component is geometrically integral because $R^3-R^2Z^2+2R^2-1$ is irreducible even after passing to an algebraic closure. The canonical component is singular at infinity, but---as we will see in the next section---the obstructions to coming from an Azumaya algebra are residues associated to smooth affine points on the canonical component, so we will often slightly abuse notation and write $C$ in place of $\widetilde{C}$.
\end{rem}
\subsection{Quaternion Azumaya Algebras on Dimension \texorpdfstring{$1$}{1} Canonical Components} \label{subsec:dim1canonicalcomponents}	
Now let $\Gamma = \pi_1(S^3 \backslash K)$ for $K$ a hyperbolic knot. Write $C$ for the normalization of a canonical component of $\SL{2}{\C}$ character variety. Since this scheme has dimension $1$ and its residue fields have characteristic zero, we may ignore all the caveats in Theorem \ref{BrauerExactSequence}. Moreover, on a curve, the codimension $1$ points are just all the points on the curve except for the generic point. The authors of \cite{CRS} consider the question of whether the canonical quaternion algebra $A_{k(C)}$ extends over all of $C$. Essentially what they prove is that $A_{k(C)}$ always extends over the points that are characters of irreducible representations and over points at infinity. In the case of canonical components coming from knots in $S^3$, they further cast the residue condition at the characters of reducible representations in terms of the arithmetic of the Alexander polynomial. In particular
\begin{defn} \label{def:condition star} Let $K$ be a knot in $S^3$. If for each root $z$ of its Alexander polynomial in a fixed algebraic closure of $\Q$ and each square root $w$ of $z$, we have an equality of fields $\Q(w+w^{-1}) = \Q(w)$, then we say that $K$ (or its Alexander polynomial) satisfies condition $(\star)$.
\end{defn}
\begin{thm}[{\cite[Theorems 1.2, 1.4]{CRS}}] \label{thm:1.2OfCRS} Let $K$ be a hyperbolic knot with $\Gamma = \pi_1\left(S^3 \backslash K\right)$, and suppose that $\Delta_K$ satisfies condition $(\star)$. Then,
	\begin{enumerate}
		\item $A_{k(C)}$ comes from an Azumaya algebra in $\Br \tilde{C}$ where $\tilde{C}$ denotes the normalization of the projective closure of $C$. \\
		\item Furthermore, if the canonical component is defined over $\Q$, there exists a finite set $S_K$ of rational primes such that, for any hyperbolic Dehn surgery $N$ on $K$ with trace field $k_N$, the $k_N$-quaternion algebra $A_N$ can only ramify at real places of $k_N$ and finite places lying over primes in $S_K$. 
	\end{enumerate} 
\end{thm} 
In particular, if condition ($\star$) holds for the Alexander polynomial of the knot, then $A_{k(C)}$ extends over a smooth, projective model of $C$ and is hence a quaternion Azumaya algebra. In view of the above theorem, we say that $K$, $\Delta_K(t)$, and $A_{k(C)}$ are \textbf{Azumaya positive} if condition $(\star)$ holds and \textbf{Azumaya negative} if not. \par
Let us comment on the connection between the Alexander polynomial and the question of extending $A_{k(C)}$. If $z$ is a root of the Alexander polynomial and $w$ is a square root of $z$, then condition ($\star$) says that $\Q(w+w^{-1})=\Q(w)$. In fact $\Q(w+w^{-1})$ is the residue field for the character of a reducible representation $\chi_{\rho}$, and $\Q(w)$ is the extension of $\Q(w+w^{-1})$ obtained by adjoining the residue of $A_{k(C)}$ at $\chi_{\rho}$. So if the fields are equal, the residue is trivial and the result follows. To be precise, in the case of a quaternion algebra like $A_{k(C)}$, its residue at any point $x$ belongs to the Galois cohomology group $H^1(k(x), \Z/2\Z)$. This group classifies (at most) quadratic extensions of $k(x)$ and is isomorphic to $k(x)^{*}/k(x)^{*^2}$ by Kummer theory. What is shown in \cite{CRS} is that the at most quadratic extension at the character of a reducible representation $\chi_{\rho}$ is precisely $\Q(w)/\Q(w+w^{-1})$, so if there is an equality of these fields, then the residue must be trivial. We note that their method of proof makes use of the tame symbol, which gives a relatively easy way to compute residue homomorphisms in this context. Given any pair of elements $\alpha, \beta \in k(C)$, the tame symbol of the quaternion algebra $\HilbertSymbol{\alpha}{\beta}{k(C)}$ at $x \in C$ is (see \cite[Theorem 3.1.(8)]{CRS})
	\[
	(-1)^{\ord_x(\alpha)\ord_x(\beta)}\beta^{\ord_x(\alpha)}/\alpha^{\ord_x(\beta)},
	\]
	where this is understood as an element of $k(x)^{*}/k(x)^{*^2}$. The point is that when the characteristic of $k(x)$ is not $2$, then this agrees with the residue. That is, the residue at $x$ is trivial if and only if this tame symbol represents $1$ in $k(x)^{*}/k(x)^{*^2}$. \par
	In particular, the Brauer class of
	\[
	\HilbertSymbol{I_g^2-4}{I_{[g,h]}-2}{k(C)}.	
	\]
	can only have nontrivial residue when $I_g=\pm 2$ or $I_{[g,h]}=2$. We note that $I_{[g,h]}=2$ corresponds to the character of reducible representations, and it turns out to account for all the nontrivial residues. This is proved in \cite[Proposition 4.1]{CRS}.
\subsection{Calculations for \texorpdfstring{$K = 7_4$}{K = 7 4}} \label{subsec:SevenFourAzumaya}
In this subsection our goal is to show how the residues may be calculated either directly or by using the Alexander polynomial, so let us now specialize for the remainder of the section to $K = 7_4$, $\Gamma = \pi_1(S^3 \backslash K)$, and $C$ the canonical component of the $\SL{2}{\C}$-character variety of $\Gamma$. Let us make some easy observations about $\Delta_K(t) = 4t^2 - 7t + 4$, the Alexander polynomial of $7_4$. Its roots are $(7 \pm \sqrt{-15})/8$, so the square roots $w$ of its roots are $\pm\frac{\sqrt{15}}{4}\pm\frac{i}{4}$. From this description it is clear that $\Q(w) = \Q(\sqrt{15},i)$ for each value of $w$. Also note that $w^{-1}=\overline{w}$, so $\Q(w+w^{-1}) = \Q(\sqrt{15})$ for each value of $w$. We will shortly see these fields emerge in calculating the tame symbol at characters of reducible representations. For now, note that these calculations show that $7_4$ does not satisfy condition $(\star)$. \par
Our curve $C$ is given by the vanishing of $R^3-R^2Z^2+2R^2-1$. As we mentioned at the end of the previous section all nontrivial residues occur at characters of reducible representations, that is, when $R=2$, and at such characters we compute the residue field as
	\[
	\Q[R,Z]/(R^3-R^2Z^2+2R^2-1,R-2) \cong \Q[Z]/(-4Z^2+15) \cong \Q[Z]/(Z^2-15) \cong \Q(\sqrt{15}).
	\] 
Not coincidentally, the residue field here is $\Q(\sqrt{15})$. The tame symbol becomes 
	\[
	\dfrac{1}{\left(\dfrac{\sqrt{15}}{2}\right)^2-4} = -4 = -1 \in \Q(\sqrt{15})^{\ast}/{\Q(\sqrt{15})^{\ast}}^2
	\]
Via Kummer theory, we identify $\Q(\sqrt{15})^{\ast}/{\Q(\sqrt{15})^{\ast}}^2$ with quadratic extensions of $\Q(\sqrt{15})$. So in our case, the class of $-1$ corresponds to the quadratic extension $\Q(\sqrt{15},i)/\Q(\sqrt{15})$. In view of the exact sequence in Theorem \ref{BrauerExactSequence}, this shows that $A_{k(C)}$ is not in the image of $\Br C \rightarrow \Br k(C)$. In other words $A_{k(C)}$ does not extend to an Azumaya algebra. Then Theorem \ref{thm:CRS-Harari} says that the quaternion algebras obtained by specializing $A_{k(C)}$ at points in the character variety ramify at primes lying above infinitely many distinct rational primes. However we do not know that these representations are geometrically interesting just from Harari's work. In fact this setup leaves open the possibility that there exists a finite set $S_K$ for $7_4$ as in the statement of Theorem \ref{thm:1.2OfCRS}. Our result shows that even when one restricts to points corresponding to the characters of $(d,0)$ hyperbolic Dehn surgery, there is still no such finite set $S_K$.
\section{Proof of Theorem \ref{maintheorem}} \label{maintheoremproofsection}
Our goal is to understand the specializations of the canonical quaternion algebra at Dehn surgery points. We already have a fairly explicit description by combining Lemma \ref{CRS:2.5} with Lemma \ref{functionfieldhilbertsymbol}. Indeed, we have that specifying a point $(r,Z)$ gives a field $k_{\rho}$ and a quaternion algebra over that field given by the Hilbert symbol
\begin{equation} \label{HilbertSymbolSpecialization}
\HilbertSymbol{-r^3+4r^2-4r-1}{-r}{k_{\rho}}.
\end{equation} \par
We now give a description of the ramification. It is stated purely algebraically, but in our applications the field $k$ will be the trace field of $(d,0)$ surgeries, and $r$ will be the corresponding coordinate coming from the character variety.
\begin{prop}\label{ramification1}
Let $r$ be an algebraic integer, $k$ a finite extension of $\Q$ containing $r$, and $\mathcal{O}$ the ring of integers of $k$. Let $N_{k/\Q}(r) = \pm p_1^{d_1} \cdot \cdot \cdot p_m^{d_m}$ be the prime factorization of the field norm of $r$ in $k/\Q$. For each $p_i \equiv 3 \Mod{4}$ with $d_i$ odd, there is a prime ideal $\mathfrak{p}_i  \subseteq \mathcal{O}$ containing $r$ and lying above $p_i$ such that 
\[
\HilbertSymbol{-r^3+4r^2-4r-1}{-r}{k_{\mathfrak{p}_i}}
\]
is a division algebra, where $k_{\mathfrak{p}_i}$ denotes $k$ completed at $\mathfrak{p}_i$.
\end{prop}
The above lemma is purely algebraic, but we will apply it when $k$ is the trace field of a $(d,0)$ surgery. In this setting $r$ will specialize to the algebraic number appearing in the lower left entry of $\rho(b)$ (with the notation of Subsection \ref{subsec:character varieties}) at the character of a Dehn surgery. To apply this proposition, we write $r_d$ for a root of the polynomial obtained by specializing the character variety (with coordinates $r$ and $Z$) to $Z=2\cos(2\pi/d)$. Write $q_d(r)$ for this polynomial. Note that $(d,0)$ hyperbolic Dehn surgery points are obtained by this specialization so that $\Q(r_d, \zeta_d + \zeta_d^{-1}) = k_d$ is the trace field at the $(d,0)$ surgery. We have
\begin{equation} \label{littlersurgerycubic}
\begin{aligned}
q_d(r) = r^3 + \left(6-\zeta_d^2-\zeta_d^{-2}\right)r^2+\left(12-4\zeta_d^2-4\zeta_d^{-2}\right)r-\left(4\left(\zeta_d^2+\zeta_d^{-2}\right)-7\right).
\end{aligned}
\end{equation}
In this notation $r_d$ is a root of $q_d$ and is an algebraic integer. Of course, if $q_d(r)$ is not irreducible, then $r_d$ is not well-defined. Even when $q_d(r)$ is irreducible, $r_d$ is only defined up to Galois conjugation; however, the ramified residue characteristics will not depend on the choice of Galois conjugate. For irreducibility we have
\begin{lem}\label{irreducibility} Let $\zeta_d$ be a primitive $d$th root of unity for $d \in \Z_{\geq 1}$ odd. The polynomial $q_d(r)\in \Q\left(\zeta_d+\zeta_d^{-1}\right)[r]$ is irreducible.
\end{lem}
We prove this lemma using a result of \cite{CalegariMorrisonSnyder} on real cyclotomic integers. Irreduciblity also allows us to compute the norm of $r_d$. In particular, if we let $k_d$ be the field generated by $r_d$ and $\zeta_d+\zeta_d^{-1}$ (so that $k_d$ is the trace field of the $(d,0)$ surgery), then the relative field norm $N_{k_d/\Q(\zeta_d+\zeta_d^{-1})}(r_d)$ is just the negative of the constant term of $q_d(r)$, namely $c_d := 4\left(\zeta_d^2+\zeta_d^{-2}\right)-7$. So then the absolute field norm of $r_d$ is equal to $N_{\Q(\zeta_d+\zeta_d^{-1})/\Q}(c_d)$. Then to apply Proposition \ref{ramification1}, we want to find $d$ such that the factorization of $N_{\Q(\zeta_d+\zeta_d^{-1})/\Q}(c_d)=N_{k_d/\Q}(r_d)$ contains prime divisors congruent to $3$ modulo $4$ an odd number of times. Such prime divisors imply the existence of a prime above them at which the canonical quaternion algebra is ramified by Proposition \ref{ramification1}. We summarize this as
\begin{prop} \label{rephrasedramification} Let $d \geq 3$ be an odd positive integer. Let $k_d$ and $A_d$ be respectively the trace field and the canonical quaternion algebra associated to the $(d,0)$ surgery. Let $p$ be a positive rational prime such that
\begin{enumerate}
	\item $p \equiv 3 \Mod{4}$ and
	\item $p$ divides $N_{\Q\left(\zeta_d+\zeta_d^{-1}\right)/\Q}\left(c_d\right)$ an odd number of times,
\end{enumerate}
then there is a finite place $\mathfrak{p}$ of $k_d$ lying above $p$ such that $A_d$ is ramified at $p$.
\end{prop}
To apply this to proving Theorem \ref{maintheorem}, we prove that infinitely many rational primes $p$ satisfy the hypotheses of Proposition \ref{rephrasedramification}. In particular, we prove
\begin{prop} \label{primeset} Let $U$ be the set of positive rational primes with
\begin{enumerate}
	\item If $p \in U$, then $p\equiv 3 \Mod{4}$,
	\item If $p \in U$, then $p$ divides $N_{k_d/\Q}\left(r_d\right)$ for some $d \in \Z_{\geq 3}$ an odd number of times.
	\end{enumerate} 
Then $U$ is infinite.
\end{prop}
Then Theorem \ref{maintheorem} follows by noting that $U \subseteq S$ where $U$ is as in the statement of Proposition \ref{primeset} and $S$ is as in the statement of Theorem \ref{maintheorem}.
\section{Proofs of Lemmas and Propositions} \label{proofs}
In this section we record the proofs of the lemmas appearing in the proof of Theorem \ref{maintheorem}
\subsection{Irreducibility}
First we prove Lemma \ref{irreducibility}. The key ingredient is the following result of \cite{CalegariMorrisonSnyder}.
\begin{thm}[{\cite[Theorem 1.0.5]{CalegariMorrisonSnyder}}] Let $\alpha \in \Q(\zeta)$ be a real algebraic integer in some cyclotomic extension of the rationals. Let $\house{\alpha}$ denote the largest absolute value of all conjugates of $\alpha$. If $\house{\alpha} \leq 2$, then $\house{\alpha} = 2\cos\left(\pi/n\right)$ for some integer $n$. If $2 \leq \house{\alpha} < 76/33$, then $\house{\alpha}$ is one of the following five numbers:
	\[
	\begin{aligned}
	\dfrac{\sqrt{7}+\sqrt{3}}{2} &= 2.188901059\dots, \\
	\sqrt{5} &= 2.236067977\dots, \\
	1+2\cos\left(2\pi/7\right) &= 2.246979602\dots, \\
	\dfrac{1+\sqrt{5}}{\sqrt{2}} &= 2\cos\left(\pi/20\right)+2\cos\left(9\pi/20\right) = 2.288245611\dots, \\
	\dfrac{1 + \sqrt{13}}{2} &= 2.302775637\dots
	\end{aligned}
	\]	
	\end{thm}
To apply this result, it will be easier to work with $p_d(R) \eqdef q_d(R-2)$, so that
\begin{equation}\label{surgerycubic}
\begin{aligned}
p_d(R) &= R^3-R^2\left(\zeta_d+\zeta_d^{-1}\right)^2+2R^2-1  \\
&= R^3-\left(\zeta_d^2+\zeta_d^{-2}\right)R^2-1.
\end{aligned}
\end{equation}
The basic idea is to prove that any root of $p_n(R)$ must lie in an interval where there are only finitely many cyclotomic integers. 
\begin{lem}\label{cubicbound}
	Let $a \in [-2,2]$. Then the absolute value of the largest real root of $R^3-aR^2-1$ is less than $2.21$.
\end{lem}
\begin{proof}
Figure \ref{fig:largestrealroot} is a graph of the absolute value of the largest real root of $x^3-ax^2-1$ pictured as a function of $a$ for $a \in [-2,2]$. \par
	\begin{figure}[h]
			\includegraphics{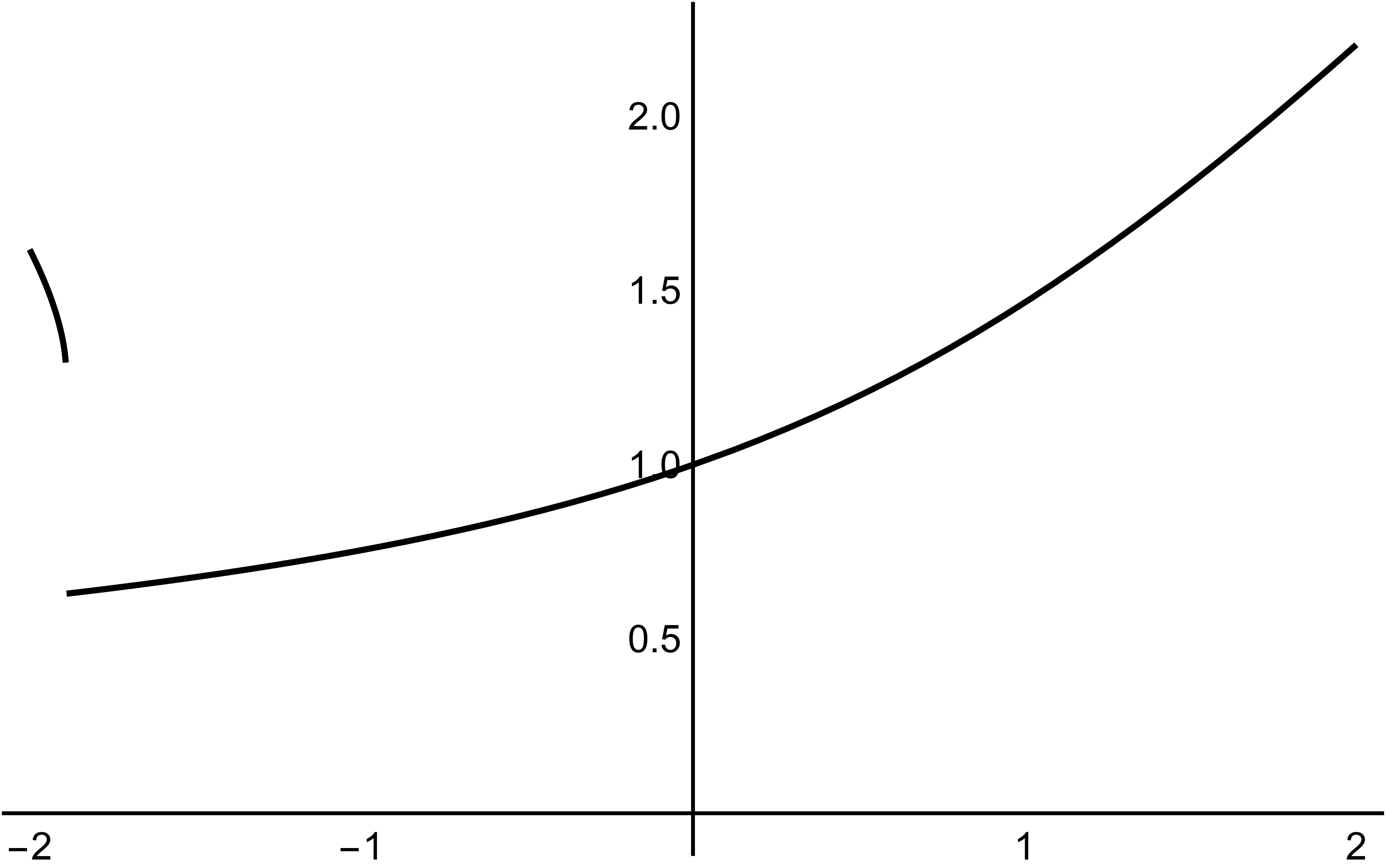}
			\caption{The largest real root of $R^3-aR^2-1$ as a function of $a \in [-2,2]$.}
			\label{fig:largestrealroot}
	\end{figure}
The right end point is the real root of $R^3-2R^2-1$ and is approximately $2.20556943040059$.
\end{proof}
\begin{rem}
The discontinuity in Figure \ref{fig:largestrealroot} comes from the fact that the discriminant is zero for $a \approx -1.88988$. For larger values of $a$ there is exactly one real root, and for smaller values there are three real roots.
\end{rem}
	Now we may prove Lemma \ref{irreducibility}.
	\begin{proof}[Proof of Lemma \ref{irreducibility}]
		Since $p_n(R)$ is of degree $3$, it suffices to show that $p_n(R)$ has no root in $\Q\left(\zeta_n+\zeta_n^{-1}\right)$. Suppose that $p_n(R)$ has a root $\alpha$. Then after Galois conjugating $p_n(R)$, we may assume it is the largest among its Galois conjugates. That is, $\alpha = \house{\alpha}$. Indeed, from Figure \ref{fig:largestrealroot}, it is clear that to obtain a root that is largest in complex absolute value among its Galois conjugates, we must choose the largest Galois conjugate of $\zeta_n^2+\zeta_n^{-2}$. This Galois conjugate is the real number $2\cos\left(4\pi/n\right)$. By Lemma \ref{cubicbound} and \cite[Theorem 1.0.5]{CalegariMorrisonSnyder}, we then have that $\alpha = \dfrac{\sqrt{7}+\sqrt{3}}{2}$. However, for $n \geq 45$, the largest real root of (the Galois conjugates of) $p_n(R)$ is greater than $\dfrac{\sqrt{7}+\sqrt{3}}{2}$. For $n = 43$, the largest real root is approximately $2.18763964834393$, which in particular is smaller than $\dfrac{\sqrt{7}+\sqrt{3}}{2}$.\par
		For $n \leq 41$, we may use a software package to verify that each of those polynomials are irreducible.
	\end{proof}
	\begin{rem} We have $p_4(R) = R^3+2R^2-1 = (R+1)(R^2+R+1)$, and $p_8(R) = R^3-1 = (R-1)(R^2+R+1)$, but $p_n(R)$ is in fact irreducible for all other even values of $n$.
	\end{rem}
\subsection{Ramification}
Next we establish the ramification behavior of the canonical quaternion algebra that we will use to produce the infinite set of primes in the statement of Theorem \ref{maintheorem}. We now prove Proposition \ref{ramification1}.
	\begin{proof}[Proof of Proposition \ref{ramification1}]
		To fix notation, let $p$ be a rational prime appearing in the factorization of $N_{k/\Q}(r)$ to an odd power, $d$. Suppose also that $p \equiv 3 \Mod{4}$. Then we know that there are prime ideals $\mathfrak{p}_1, \cdots, \mathfrak{p}_{m'}$ of $\mathcal{O}$ that $r$ belongs to. In fact, for some such prime ideal $\mathfrak{p}_i$, we have that $r$ belongs to $\mathfrak{p}_i^{g}$ but not $\mathfrak{p}_i^{g+1}$ for some odd integer $g$. To see this, first write $(r) \subseteq \mathfrak{p}_1^{g_1}\cdots\mathfrak{p}_{m'}^{g_{m'}}$, where each $\mathfrak{p}_i$ lies above $p$, $\mathfrak{p}_i \neq \mathfrak{p}_j$ if $i \neq j$,  and each power $g_i$ is maximal. Furthermore suppose that every prime ideal lying above $p$ and containing $r$ appears in this factorization. Then we take the ideal norm $\mathfrak{N}\left(\mathfrak{p}_1^{g_1}\cdots\mathfrak{p}_{m'}^{g_{m'}}\right) = p^{g_1f_1}\cdots p^{g_{m'}f_{m'}}=p^{\sum_{i=1}^{m'}g_if_i}=p^d$, where $f_i$ is the residue class degree. If each $g_if_i$ were even, then $\sum_{i=1}^{m'}g_if_i$ would be an even integer, but $d$ is odd, so there is some $i$ with $g_i$ and $f_i$ both odd. Then, we may assume after scaling $r$ by squares that $r \in \mathfrak{p}\backslash \mathfrak{p}^2$ where $\mathfrak{p}$ has odd residue class degree $f$. \par
		Now applying Theorem \ref{MRRamification}, we have that
		\[
		\HilbertSymbol{-r^3+4r^2-4r-1}{-r}{k_{\mathfrak{p}}}
		\]
		is ramified if and only if $-r^3 + 4r^2 - 4r - 1$ is not a square modulo $\mathfrak{p}$. Since $r \in \mathfrak{p}$, this is equivalent to asking whether $-1$ is a square modulo $\mathfrak{p}$. Indeed, $-1$ is not a square in the finite field $\mathbf{F}_{p^f}$ if and only if $p \equiv 3 \Mod{4}$ and $f$ is odd. This can be seen via Jacobi symbols, for example.
	\end{proof}
	Recall that Lemma \ref{functionfieldhilbertsymbol} gives a description of the quaternion algebra over the function field of the canonical component. Specializing $r$ to $r_d$ and taking the ground field to be the trace field of the $(d,0)$ surgery, we obtain
	\[
	\HilbertSymbol{-r_d^3+4r_d^2-4r_d-1}{-r_d}{k_d}.
	\] 
Moreover, $N_{\Q\left(\zeta_d+\zeta_d^{-1}\right)/\Q}\left(c_d\right)$ is the norm of $r_d$. So Proposition \ref{ramification1} says that the quaternion algebra associated to $(d,0)$ hyperbolic Dehn surgery is ramified at some prime lying above any rational prime divisor of $N_{\Q(\zeta_d+\zeta_d^{-1})/\Q}(c_d)=N_{k_d/\Q}(r_d)$ that appears to an odd power and is congruent to $3 \Mod{4}$. This observation proves Proposition \ref{rephrasedramification}. Then we are left to prove that we can actually find infinitely many distinct such rational primes as $d$ varies. This is the content of Proposition \ref{primeset}, which we now turn to proving. \par
We wish to understand when $N_{\Q\left(\zeta_d+\zeta_d^{-1}\right)/\Q}\left(c_d\right)$ has a prime divisor $p$ that is congruent to $3 \Mod{4}$ and divides $N_{\Q\left(\zeta_d+\zeta_d^{-1}\right)/\Q}\left(c_d\right)$ a strictly odd number of times. In view of Proposition \ref{rephrasedramification}, this will say that $p \in T$ where $T$ is as in the statement of Theorem \ref{maintheorem}. We will accomplish this by showing that for certain $d$ that $\abs{N_{\Q\left(\zeta_d+\zeta_d^{-1}\right)/\Q}\left(c_d\right)} \equiv 3 \Mod{4}$. The lemma we now prove basically says that $N_{\Q\left(\zeta_d+\zeta_d^{-1}\right)/\Q}\left(c_d\right)$ is always $1 \Mod{4}$.
\begin{lem}\label{normisonemodfour} 
	Let $d \geq 3$ be an odd positive integer, $\zeta_d$ a primitive $d$th root of unity, and $c_d = 4(\zeta_d^2+\zeta_d^{-2})-7$. Then $N_{\Q\left(\zeta_d+\zeta_d^{-1}\right)/\Q}\left(c_d\right) \equiv 1 \Mod{4}.$
\end{lem}
	\begin{proof}Observe that $c_d \equiv 1 \Mod{4}$, so the product over the Galois conjugates is also $1 \Mod{4}$. 
	\end{proof}
Then, if we want the \textit{absolute value} of $N_{\Q\left(\zeta_d+\zeta_d^{-1}\right)/\Q}\left(c_d\right)$ to be $3\Mod{4}$, we need the norm itself to be negative. Determining exactly which $d$ make $N_{\Q\left(\zeta_d+\zeta_d^{-1}\right)/\Q}\left(c_d\right)$ negative turns out to be somewhat difficult.
\subsection{The Sign of the Norm} We start with a lemma that is visibly not about signs, but will later give us some information.
	\begin{lem} \label{absolutevaluesareequal} Let $\beta = \dfrac{i+\sqrt{15}}{4}$. Then
		\[
		\abs{\prod_{d \mid n} N_{\Q\left(\zeta_d+\zeta_d^{-1}\right)/\Q}\left(c_d\right)} = \abs{2^{n+1}\Im\left(\beta^n\right)}.
		\]
	\end{lem}
	\begin{proof}
		We prove that
		\[
		\left(\prod_{d \mid n} N_{\Q\left(\zeta_d+\zeta_d^{-1}\right)/\Q}\left(c_d\right)\right)^2 = \left(2^{n+1}\Im\left(\beta^n\right)\right)^2.
		\]
		Consider the function $f: \Z_{\geq 1} \rightarrow \Z_{\geq 0}$ defined by $f(n)^2 = \res_x(x^n-1,4x^4-7x^2+4)$. It's not completely obvious that $f(n)^2$ is a square integer. For now, however, note that $\beta$ is a root of $4x^4-7x^2+4$. The other roots are $-\beta$ and $\pm \overline{\beta}$. By the multiplicative property of the resultant we have 
		\[
		\begin{aligned}
		f(n)^2 &= \res_x(x^n-1, 4(x-\beta)(x+\beta)(x-\overline{\beta})(x+\overline{\beta})) \\
		&= \res_x(x^n-1,2(x-\beta)(x+\overline{\beta})\res_x(x^n-1, 2(x+\beta)(x-\overline{\beta})).
		\end{aligned}
		\]
		If we write $g(n) = \res_x\left(x^n-1,2(x-\beta)(x+\overline{\beta})\right)$ and $\gamma = 2\beta$, we may compute
		\[
		\begin{aligned}
		g(n) &= \res_x\left(x^n-1,2(x-\beta)(x+\overline{\beta})\right) \\
		&=2^n\left(\beta^n-1\right)(-\overline{\beta}^n-1) \\
		&=2^n\left(\overline{\beta}^n-\beta^n\right) \\
		&= \overline{\gamma}^n-\gamma^n.
		\end{aligned}
		\]
		Note that since $\gamma$ is integral over $\Z$ (its minimal polynomial is $x^4-7x^2+16$), the above calculation shows that $g(n)$ is as well. Moreover, the fixed field of the automorphism $\Q(\beta) \rightarrow \Q(\beta)$ determined by $\gamma \mapsto -\overline{\gamma}$ is $\Q(i)$. The easiest way to see this is to note that this automorphism fixes $i$ and takes $\sqrt{15}$ to $-\sqrt{15}$. It follows that $g(n) \in \Z[i]$. We can also calculate that
		\[
		\begin{aligned}
		g(n) &= 2^n\left(\overline{\beta}^n-\beta^n\right) \\
		&= -2^{n+1}\Im(\beta^n)i.
		\end{aligned}
		\]
		It then follows that $2^{n+1}\Im(\beta^n) \in \Z$. On the hand, we can compute that the other factor of the original resultant (that is, of $f(n)^2$) is
		\[
		\begin{aligned}
		\res_x\left(x^n-1, 2(x+\beta)(x-\overline{\beta})\right) &= 2^n(-\beta^n-1)(\overline{\beta}^n-1) \\
		&= 2^{n+1}\Im(\beta^n)i.
		\end{aligned}
		\]
		It follows that $f(n)^2 = 4^{n+1}\Im(\beta^n)^2 = \left(2^{n+1}\Im(\beta^n)\right)^2$. Since $2^{n+1}\Im(\beta^n) \in \Z$, it follows that $f(n)^2$ is in fact a positive square integer, and its positive square root is given by $\pm 2^{n+1}\Im(\beta^n)$.
		\item  Writing $\Phi_d(x)$ for the $d$th cyclotomic polynomial, we have that $x^n-1 = \prod_{d \mid n} \Phi_d(x)$, so
		\[
		\begin{aligned}
		\res_x\left(x^n-1, 4x^4-7x^2+4\right) &= \prod_{d \mid n}\res_x\left(\Phi_d(x), 4x^4-7x^2+4\right) \\
		&= \prod_{d \mid n}N_{\Q(\zeta_d)/\Q}\left(4\zeta_d^4-7\zeta_d^2+4\right) \\
		\end{aligned}
		\]
		Now consider $c_d$ first as an element of $\Q(\zeta_d)$. However note that $\zeta_d^2c_d = 4\zeta_d^4-7\zeta_d^2+4$, and $N_{\Q(\zeta_d)/\Q}\left(\zeta_d\right)=1$, so $N_{\Q(\zeta_d)/\Q}\left(c_d\right) = N_{\Q(\zeta_d)/\Q}\left(4\zeta_d^4-7\zeta_d^2+4\right)$. But $N_{\Q(\zeta_d)/\Q}\left(c_d\right) = \left(N_{\Q(\zeta_d+\zeta_d^{-1})/\Q}\left(c_d\right)\right)^2$, since $\Q(\zeta_d)$ is a quadratic extension of $\Q(\zeta_d+\zeta_d^{-1})$. We summarize this as
		\[
		\res_x\left(x^n-1, 4x^4-7x^2+4\right) = \left(\prod_{d \mid n}N_{\Q(\zeta_d+\zeta_d^{-1})/\Q}\left(c_d\right)\right)^2,
		\]
		so
		\[
		f(n)^2 = \left(\prod_{d \mid n}N_{\Q(\zeta_d+\zeta_d^{-1})/\Q}\left(c_d\right)\right)^2.
		\]
	\end{proof}
	Next we determine the residue class of $2^{n+1} \Im\left(\beta^n\right)$ for odd $n$.
	\begin{lem} \label{betaresidueclasses}
		For all $n \in \Z_{\geq 1}$ odd, we have
		\[
		2^{n+1}\Im(\beta^n) \equiv \begin{cases}
		1 \Mod{4} & \text{ if } n \equiv 1 \Mod{4} \\
		3 \Mod{4} & \text{ if } n \equiv 3 \Mod{4}.
		\end{cases}	
		\]
	\end{lem}
	\begin{proof}
		As in the proof of Lemma \ref{absolutevaluesareequal}, if we write $\gamma = 2\beta$, then $\gamma$ is an algebraic integer, and
		$2^{n+1}\Im(\beta^n)i = \gamma^n-\overline{\gamma}^n.$ Then, we have $i\left(\overline{\gamma}^n-\gamma^n\right)=2^{n+1}\Im(\beta^n)$. We may compute that
		\[
		i\left(\overline{\gamma}^n-\gamma^n\right) \equiv \begin{cases}
		1 \Mod{4} & \text{ if } n \equiv 1 \Mod{4} \\
		3 \Mod{4} & \text{ if } n \equiv 3 \Mod{4}.
		\end{cases}	
		\]
		It's worth pointing out that this reduction is $\mathcal{O}_K \twoheadrightarrow \mathcal{O}_K/4\mathcal{O}_K$ where $\mathcal{O}_K$ is the ring of integers of the field $K=\Q(\gamma)$.
	\end{proof}
	Combining Lemmas \ref{absolutevaluesareequal}, \ref{betaresidueclasses}, and \ref{normisonemodfour} gives
	\begin{lem} \label{signsareequal}
		Let $n\geq 5$ be a positive integer such that $n\equiv1\Mod{4}$ and $\beta = \dfrac{\sqrt{15}+i}{4}.$ Then
		\begin{equation}
		\prod_{d \mid n} N_{\Q\left(\zeta_d+\zeta_d^{-1}\right)/\Q}\left(c_d\right) = 2^{n+1} \Im\left(\beta^n\right).
		\end{equation}
	\end{lem}
	\begin{proof}
		By Lemma \ref{absolutevaluesareequal}, we have an equality of absolute values. So we just have to check that the signs are equal. It suffices to check that they are both congruent to $1$ modulo $4$. Indeed,  each $N_{\Q(\zeta_d+\zeta_d^{-1})/\Q}(c_d) \equiv 1\Mod{4}$ by Lemma \ref{normisonemodfour}, and when $n\equiv 1 \Mod{4}$, so is $2^{n+1}\Im(\beta^n)$ by Lemma \ref{betaresidueclasses}. 
	\end{proof}
We are left to compute the sign of $\Im(\beta^n)$ for integers $n$. Since $\beta$ is on the unit circle in the complex plane, we may write $\beta = e^{2\pi i x}$ so that $\beta^n = e^{2\pi i nx}$. Then $\Im(\beta^n) < 0$ if and only if $\dfrac{nx}{2\pi}$ is greater than $1/2 \Mod{1}$. The following result of Furstenberg allows us to easily prove the existence of such $n$. Before stating it, we recall that a multiplicative semigroup of the integers is called \textbf{lacunary} if it consists of powers of a single integer and \textbf{non-lacunary} otherwise.
\begin{thm}[{\cite[Theorem IV.1]{FurstenbergNonlacunary}}] \label{thm:nonlacunary}
If $\Sigma$ is a non-lacunary semigroup of integers and $\alpha$ is irrational, then $\Sigma \alpha$ is dense modulo $1$.
\end{thm}
We remark that the non-lacunary semigroups we consider are those of the form
\[
\{l_1^{r_1}l_2^{r_2}\cdots l_m^{r_m} \> | \> l_i \text{ prime, } l_i \equiv 1 \Mod{4} \}
\]
with $m \geq 2$.

\subsection{Proof of Proposition \ref{primeset}}
Let us briefly say where we are going. Recall our notation from Section \ref{maintheoremproofsection} that $r_d$ is the coordinate appearing in the Hilbert symbol for the $(d,0)$ hyperbolic Dehn surgery. Proposition \ref{rephrasedramification} reduced the ramification of the quaternion algebra to finding rational prime divisors on the norm of $r_d$, and it is technically simpler to work with $c_d = 4(\zeta_d^2+\zeta_d^{-2})-7$, which has the property that $N_{\Q(\zeta_d+\zeta_d^{-1})/\Q}\left(c_d\right)=N_{k_d/\Q}(r_d)$. Using Furstenberg's Theorem \ref{thm:nonlacunary}, we wish to construct a sequence $(d_i)$ such that the set of residue characteristics of ramified places of the $(d_i,0)$ surgeries form an infinite set. 
In view of Proposition \ref{rephrasedramification}, this amounts to finding infinitely many distinct rational prime divisors of $N_{\Q(\zeta_d+\zeta_d^{-1})/\Q}\left(c_d\right)=N_{k_d/\Q}(r_d)$ which are equivalent to $3 \Mod{4}$ and appear to an odd power in the prime factorization of the norm. Such a sequence will be constructed in Lemma \ref{divisorsequence}. \par
Our first goal is to prove
\begin{lem} \label{primefinitelymanytimes} Let $p$ be a rational prime. Then, $p$ divides $N_{\Q(\zeta_d+\zeta_d^{-1})/\Q}\left(c_d\right)=N_{k_d/\Q}(r_d)$ for only finitely many values of $d$ coprime to $p$.
\end{lem}
Let us recall a fact from basic number theory. See, e.g., \cite[Proposition 1.10.3]{NeukirchANT}. 
\begin{prop}\label{finitedegrees} Let $p$ be a rational prime and $d$ an integer such that $p \nmid d$. Let $\mathbf{F}_p(\zeta_d)$ be the field obtained by adjoining a primitive $d$th root of unity to the finite field with $p$ elements, $\mathbf{F}_p$. Then this extension is cyclic of degree equal to the multiplicative order of $p \Mod{d}$.
\end{prop} 
The next lemma follows from well-known facts, but we include a proof for completeness.
\begin{lem}\label{pm1modl} Let $d \in \Z_{\geq 3}$ be odd and $\zeta_d$ a primitive $d$th root of unity. Then the prime divisors of $N_{\Q(\zeta_d)/\Q}(c_d)$ not dividing $d$ have multiplicative order modulo $d$ equal to $1$ or $2$.
\end{lem}
\begin{proof}
	Let $\mathfrak{p}$ be a prime ideal of $\Q(\zeta_d)$ lying above the rational prime $p$ such that $c_d=4\zeta_d^4-7\zeta_d^2+4$ belongs to $\mathfrak{p}$. Note that since $d$ is odd, $c_d$ is Galois conjugate in $\Q(\zeta_d)/\Q$ to $4\zeta_d^2-7\zeta_d+4$, so it suffices to show the lemma for this latter algebraic integer. Also suppose that $p \nmid d$. Consider the reduction map $\Z[\zeta_d] \rightarrow \Z[\zeta_d]/\mathfrak{p} \cong \mathbf{F}_p(\zeta_d).$ Note that this reduction takes $d$th roots of unity of $\Z[\zeta_d]$ bijectively onto $d$th roots of unity of $\mathbf{F}_p(\zeta_d)$ hence primitive $d$th roots of unity remain primitive. By assumption, $4\zeta_d^2-7\zeta_d+4$ is in the kernel of this map. Henceforth we write $\zeta_d$ for the image of $\zeta_d \in \Z(\zeta_d)$ under this reduction map. That is, $4\zeta_d^2-7\zeta_d+4 = 0$ in $\mathbf{F}_p(\zeta_d)$. This implies that $\{1, \zeta_d, \zeta_d^2\}$ is linearly dependent over $\mathbf{F}_p$. Since $\{1, \zeta_d, \zeta_d^2, \dots, \zeta_d^{m-1}\}$ is an $\mathbf{F}_p$-basis for $\mathbf{F}_p(\zeta_d)$ where $m$ is the degree of the extension $\mathbf{F}_p(\zeta_d)/\mathbf{F}_p$, we have that $m \leq 2$. Then by Proposition \ref{finitedegrees}, the multiplicative order of $p \Mod{d}$ is either $1$ or $2$.
\end{proof}
We may now easily prove Lemma \ref{primefinitelymanytimes}.
\begin{proof}[Proof of Lemma \ref{primefinitelymanytimes}]
	Any prime divisor $p \in \Z_{\geq 2}$ of $N_{\Q(\zeta_d+\zeta_d^{-1})}(c_d)$ is either divides $d$ itself or has multiplicative order equal to $1$ or $2$ modulo $d$. Any prime $p$ has multiplicative order modulo $d$ equal to $1$ or $2$ for only finitely many values of $d$ (e.g. take $d > p^2$).  We conclude that a given prime $p$ divides $N_{\Q(\zeta_d+\zeta_d^{-1})}(c_d)$ for finitely many values of $d$.
\end{proof}
\begin{rem} The sequence of $(d,0)$ surgeries we construct have $d$ only divisible by primes congruent to $1\Mod{4}$, but the ramified primes we find are all $-1\Mod{4}$, so there is no issue of finding the same prime infinitely often as a divisor of the surgery coefficients.
\end{rem}
\begin{lem}\label{auxilarysequnce} Let $\Sigma$ be a non-lacunary semigroup of integers of the form $\{l_1^{r_1}l_2^{r_2}\cdots l_m^{r_m} \> | \> l_i \text{ prime, } l_i \equiv 1 \Mod{4} \}$. Then there exists a sequence $\left(n_i\right)_{i=1}^{\infty}$ of positive integers such that
	\begin{enumerate}
		\item Each $n_i$ is divisible only by the primes $\{l_j\}$ appearing in $\Sigma$,
		\item If $j > i$, $n_i \mid n_j$.
		\item If $i \neq j$, then $n_i \neq n_j$,
		\item If $i$ is even, then $2^{n_i+1}\Im\left(\beta^{n_i}\right) > 0$, and
		\item If $i$ is odd, then $2^{n_i+1}\Im\left(\beta^{n_i}\right) < 0$.
	\end{enumerate}
\end{lem}
That the sequence is built out of powers of primes from $\Sigma$ guarantees that each $n_i \equiv 1 \Mod{4}$, and in fact each divisor of $n_i$ must also be congruent to $1 \Mod{4}$.
\begin{proof}
	We construct such a sequence by repeatedly applying Theorem \ref{thm:nonlacunary}. Let $x_1$ be defined by $e^{2\pi i x_1} = \beta$. Note that $x_1$ is irrational (in fact transcendental by Gelfond-Schneider). Let $n_1$ be any element of $\Sigma$ such that $n_1 x_1 > 1/2 \Mod{1}$. This implies that $2^{n_1+1}\Im(\beta^{n_1}) < 0$. To construct $n_2$, set $x_2 = n_1 x_1$ so that $e^{2 \pi i x_2} = \beta^{n_1}$. Since $x_2$ is also irrational, Theorem \ref{thm:nonlacunary} applies to $\Sigma x_2$ to prove an $m_2 \in \Sigma$ such that $m_2 x_2 < 1/2 \Mod{1}$. Set $n_2 = m_2n_1$. Note that $n_1 \mid n_2$. Proceeding in this manner constructs the desired sequence.
\end{proof}
We now extract a sequence $(d_i)_{i=1}^{\infty}$ where $d_i \mid n_i$ and $d_i$ satisfies the hypotheses of Proposition \ref{ramification1}.
\begin{lem}\label{divisorsequence} There exists a sequence $\left(d_i\right)_{i=1}^{\infty}$ of positive integers such that
	\begin{enumerate}
		\item If $i \neq j$, then $d_i \neq d_j$.
		\item For each $i$, $\abs{N_{\Q\left(\zeta_{d_i}+\zeta_{d_i}^{-1}\right)/\Q}\left(c_{d_i}\right)} \equiv 3 \Mod{4}$.
	\end{enumerate}
\end{lem}
\begin{proof}
	We use the sequence $\left(n_i\right)_{i=1}^{\infty}$ constructed in Lemma \ref{auxilarysequnce}. Let us first construct $d_1$. We have that
	\[
	2^{n_1+1}\Im\left(\beta^{n_1}\right) = \prod_{d \mid n_1} N_{\Q\left(\zeta_d+\zeta_d^{-1}\right)/\Q}\left(c_d\right),
	\]
	by Lemma \ref{signsareequal} and the fact that the $l_i$ appearing in the definition of $\Sigma$ in Lemma \ref{auxilarysequnce} are all $1\Mod{4}$. Since $\prod\limits_{d \mid n_1} N_{\Q\left(\zeta_d+\zeta_d^{-1}\right)/\Q}\left(c_d\right)$ is negative by construction, we must have that some $N_{\Q\left(\zeta_d+\zeta_d^{-1}\right)/\Q}\left(c_d\right)$ is negative. Since $N_{\Q\left(\zeta_d+\zeta_d^{-1}\right)/\Q}\left(c_d\right) \equiv 1\Mod{4}$, by Lemma \ref{normisonemodfour}, we have that $\abs{N_{\Q\left(\zeta_d+\zeta_d^{-1}\right)/\Q}\left(c_d\right)} \equiv 3 \Mod{4}$. Set this $d$ equal to $d_1$. To construct $d_2$, we first consider $n_2$. We have that
	\[
	\prod_{d \mid n_1} N_{\Q\left(\zeta_d+\zeta_d^{-1}\right)/\Q}\left(c_d\right) \bigg| \prod_{d \mid n_2} N_{\Q\left(\zeta_d+\zeta_d^{-1}\right)/\Q}\left(c_d\right),
	\]
	because $n_1 \mid n_2$. However $\prod\limits_{d \mid n_2} N_{\Q\left(\zeta_d+\zeta_d^{-1}\right)/\Q}\left(c_d\right)$ is positive, so there must be some $d_2$ such that $d_2 \nmid n_1$ but $d_2 \mid n_2$ with $N_{\Q\left(\zeta_{d_i}+\zeta_{d_i}^{-1}\right)/\Q}\left(c_{d_i}\right)$ negative. Then, as before, $\abs{N_{\Q\left(\zeta_{d_2}+\zeta_{d_2}^{-1}\right)/\Q}\left(c_{d_2}\right)} \equiv 3\Mod{4}$. Proceeding in this manner we obtain the sequence.
\end{proof}
	Now Proposition \ref{primeset} can be proved easily.
	\begin{proof}[Proof of Proposition \ref{primeset}]
		We consider the sequence $(d_i)_{i=1}^{\infty}$ of Lemma \ref{divisorsequence}. For each such $d_i$, we have $\abs{N_{\Q(\zeta_{d_i}+\zeta_{d_i}^{-1})/\Q}(c_{d_i})} \equiv 3 \Mod{4}$, so there is some prime $p$ with $p \equiv 3\Mod{4}$ that divides $\abs{N_{\Q(\zeta_{d_i}+\zeta_{d_i}^{-1})/\Q}(c_{d_i})}$ an odd number of times. By Lemma \ref{primefinitelymanytimes}, any such prime $p$ divides $\abs{N_{\Q(\zeta_{d_i}+\zeta_{d_i}^{-1})/\Q}(c_{d_i})}$ for only finitely many $i$. Since $(d_i)_{i=1}^{\infty}$ is an infinite sequence, we conclude that there must be infinitely many distinct rational primes that are congruent to $3$ modulo $4$ that divide $\abs{N_{\Q(\zeta_d+\zeta_d^{-1})/\Q}(c_{d})}$ an odd number of times where $d$ ranges over the odd positive integers. Recalling that $N_{\Q(\zeta_d+\zeta_d^{-1})/\Q}\left(c_d\right)=N_{k_d/\Q}(r_d)$ completes the proof. 
	\end{proof}
\section{Torsion Points} \label{TorsionPointsSection}
Recall from the introduction that $E$ is the elliptic curve defined by $y^2 = x^3 + 2x^2 - 1$. We use the traditional variables $x$ and $y$ for an elliptic curve, and one may specify the birational map from the canonical component which is cut out by $R^3+(2-Z^2)R^2-1=0$ either by the map $C \dashrightarrow E$, $(R,Z) \mapsto (R,RZ)$ or by the coordinate change $R=x, Z=\dfrac{y}{x}$. Theorem \ref{maintheorem2} follows from the following proposition.
\begin{prop} \label{prop:2-adic valuation} Let $E$ be the elliptic curve defined by the Weierstrass equation $y^2 = x^3 + 2x^2 - 1$. Then, excepting the $2$-torsion points, every torsion point of $E$ has $x$-coordinate equal to an algebraic integer with positive $2$-adic valuation.
\end{prop}
We now introduce the relevant topological background material to explain how Proposition \ref{prop:2-adic valuation} implies Theorem \ref{maintheorem2}.
We begin by fixing some notation and recalling results of Hatcher and Hatcher-Thurston. 
\begin{thm}[\cite{HatcherThurston85}, \cite{HatcherBoundarySlopes}] \label{HatcherThurstonBundle}
Let $K$ be a hyperbolic two-bridge knot. Then
	\begin{enumerate}
		\item $E(K)$ has no closed, embedded, essential surface.
		\item All but finitely many Dehn surgeries are non-Haken and hyperbolic.
	\end{enumerate}
\end{thm}

As noted in the introduction, the knot $7_4$ is a two-bridge knot, so Theorem \ref{HatcherThurstonBundle} implies that the exterior of the knot has no closed, embedded essential surface and all but finitely many of its surgeries are hyperbolic and non-Haken. Bass's theorem (see \cite{BassTheorem} or \cite[Section 5.2]{MR}) implies that if $N = \mathbf{H}^3/\Gamma$ is a hyperbolic surgery on $7_4$, then the traces of $\Gamma$ are algebraic integers. In particular at points corresponding the character of Dehn surgeries, the trace of a meridian (with finitely many exceptions) is an algebraic integer. However, when $Z$ is integral, the relation $R^3+(2-Z^2)R^2-1=0$ implies that $R$ is a unit. Then Proposition \ref{prop:2-adic valuation} says that $R=x$ is never a unit when $R$ is the first coordinate of a torsion point on $X$. \par
Let us briefly treat the finitely many exceptions. There are 3 boundary slopes: $0/1$, $-8/1$, and $-14/1$ (see \cite{DunfieldBoundarySlopes}). The first is not hyperbolic; the second has integral traces as one can check in Snap; the third does have non-integral traces, so we must check it directly. One may compute that $R$ has negative $2$-adic valuation at this point and hence Proposition \ref{prop:2-adic valuation} also implies that this point cannot be torsion. \par
Finally, the $2$-torsion points are not covered by \ref{prop:2-adic valuation}. The $1$-torsion is just the point at infinity, and is in particular not in the image of the birational map defined above. The nontrivial $2$-torsion consists of three points. The $y$ coordinate of each of them is $0$ and the three $x$-coordinates are the roots of $x^3+2x^2-1$. All of these roots are real, which implies that the trace field associated to the representation is real, but every finite covolume Kleinian group has nonreal trace field, so no $2$-torsion can be a hyperbolic Dehn surgery point.
\subsection{Division Polynomials for Elliptic Curves}
We now recall some basic facts and fix notation about the division polynomials associated to an elliptic curve. These polynomials will be the main tool in the proof of Proposition \ref{prop:2-adic valuation}. In this section we use the variables $x$ and $y$ to be consistent with the literature on elliptic curves, but one may convert back to traces on the canonical component with the relations $R=x$ and $Z=\dfrac{y}{x}$.
\begin{defn} For an elliptic curve $E$ defined by the Weierstrass equation $y^2+a_1xy+a_3y = x^3+a_2x^2+a_4x+a_6$, we define the following standard quantities
	\[
	\begin{aligned}
	b_2 &= a_1^2 + 4a_2, \\
	b_4 &= 2a_4 + a_1a_3, \\
	b_6 &= a_3^2 + 4a_6, \\
	b_8 &= a_1^2a_6 + 4a_2a_6 - a_1a_3a_4 + a_2a_3^2 - a_4^2, \\
	c_4 &= b_2^2 - 24b_4, \\
	c_6 &= -b_2^3 + 36b_2b_4 - 216b_6, \\
	\Delta &= -b^2b_8 - 8b_4^3 - 27b_6^2 + 9b_2b_4b_6, \\
	j &= c_4^3/\Delta.
	\end{aligned}
	\]
\end{defn}
\begin{rem} For the curve in Proposition \ref{prop:2-adic valuation}, we have
	\begin{alignat*}{4}
	a_1 &= 0, &\qquad b_2 &= 8, &\qquad c_4 &= 64, &\qquad \Delta &= 80, \\
	a_2 &= 2, & b_4 &= 0, & c_6 &= 352, & j &= \dfrac{16384}{5} = \dfrac{2^{14}}{5}, \\
	a_3 &= 0, & b_6 &= -4 & & & & \\
	a_4 &= 0, & b_8 &= -8 & & & & \\
	a_5 &= -1.
	\end{alignat*}

\end{rem}

\begin{defn} For an elliptic curve $E$ defined by the Weierstrass equation $y^2+a_1xy+a_3y = x^3+a_2x^2+a_4x+a_6$, we define two families of polynomials $\psi_n(x,y)$ and $f_n(x)$ by
	\[
	\begin{aligned}
	\psi_1 &= 1, \\
	\psi_2 &= 2y+a_1x+a_3, \\
	\psi_3 &= 3x^4 + b_2x^3 + 3b_4x^2+3b_6x + b_8. \\
	\psi_4 &= \psi_2(x,y)\cdot\left(2x^6 + b_2x^5 + 5b_4x^4 + 10b_6x^3 + 10b_8x^2 + \left(b_2b_8-b_4b_6\right)x+(b_4b_8-b_6^2)\right),
	\end{aligned}
	\] 
	and then recursively via
	\begin{align}
	\psi_{2m+1} &= \psi_{m+2}\psi_m^3-\psi_{m-1}\psi_{m+1}^3, \label{psiodd} \\
	\psi_{2}\psi_{2m} &= \psi_{m-1}^2\psi_m\psi_{m+2} - \psi_{m-2}\psi_m\psi_{m+1}^2. \label{psieven} 
	\end{align}
	We note that $\psi_n(x,y)$ is a polynomial in $x$ when $n$ is odd, and---using the relation $(2y+a_1x+a_3)^2 = 4x^3+b_2x^2+2b_4x+b_6$--- $(2y+a_1x+a_3)\psi_n(x,y)=\psi_2(x,y)\psi_n(x,y)$ is a polynomial in $x$ when $n$ is even, so we may further define
	\[
	f_n(x) = 
	\begin{cases}
	\psi_n(x,y) &  n\text{ odd}, \\
	\psi_2(x,y)\psi_n(x,y) & n \text{ even}.
	\end{cases}
	\]
	These polynomials are known as the \textbf{division polynomials}. 
\end{defn}
\begin{rem} The notation and terminology surrounding the division polynomials is not entirely standard in the literature. The definition for $\psi_n$ above is consistent with \cite[Exercise 3.7]{SilvermanEC}. The definition of $f_n$ agrees with the GP/PARI (\cite{PARI2}) function \texttt{elldivpol} so that $f_n(x)$ is the output of \texttt{elldivpol(E,n)}.
\end{rem}
\begin{rem}
	The roots of $f_n(x)$ are precisely the $x$-coordinates of the nontrivial $n$-torsion points.
\end{rem}
\begin{prop} Let $f_n$ be as above. Then $f_n$ satisfies the following recursive relations.
	\begin{enumerate}
		\item If $n=2m$, then
		\[
		f_2f_{2m} = f_m\left(f_{m-1}^2f_{m+2}-f_{m-2}f_{m+1}^2\right).
		\]
		\item If $n \equiv 1 \Mod{4}$ and we write $n=2m+1$, then
		\[
		f_2^2f_{2m+1} = f_{m+2}f_m^3-f_2^2f_{m-1}f_{m+1}^3.
		\]
		\item If $n \equiv 3 \Mod{4}$, and we write $n=2m+1$, then
		\[
		f_2^2f_{2m+1} = f_2^2f_{m+2}f_m^3-f_{m-1}f_{m+1}^3.
		\]
	\end{enumerate}
\end{prop}
\begin{proof}
	These all follow from the recursive formulas for $\psi_n$, but we include a proof since we were unable to find them in the literature. \par
	We have to treat each residue class modulo $4$ separately. So first suppose that $n \equiv 0 \Mod{4}$. That is, $n=2m$ for $m$ an even number.
	Using Equation \ref{psieven}, we have
	\[
	\begin{aligned}
	f_n(x) &= \psi_2\psi_{2m} \\
	&=\psi_m(f_{m-1}^2\psi_{m+2}-\psi_{m-2}f_{m+1}^2).
	\end{aligned}
	\]
	Note that $f_2 = \psi_2^2$, so
	\[
	\begin{aligned}
	f_2f_{2m} &= \psi_2\psi_m(f_{m-1}^2\psi_2\psi_{m+2}-\psi_2\psi_{m-2}f_{m+1}^2) \\
	&= f_m(f_{m-1}^2f_{m+2}-f_{m-2}f_{m+1}^2).
	\end{aligned}
	\]
	\par
	Now say that $n \equiv 2 \Mod{4}$ so that $n-2m$ for $m$ an odd number. Using Equation \ref{psieven} again, we obtain
	\[
	\begin{aligned}
	f_n = f_{2m} &= \psi_2\psi_{2m} \\
	&= f_m(\psi_{m-1}^2f_{m+2} - f_{m-2}\psi_{m+1}^2).
	\end{aligned}
	\]
	Then we find
	\[
	\begin{aligned}
	f_2f_{2m} &= \psi_2^2f_{2m} \\
	&= f_m\left(\left(\psi_2\psi_{m-1}\right)^2f_{m+2}-f_{m-2}\left(\psi_2\psi_{m+1}\right)^2\right) \\
	&= f_m\left(f_{m-1}^2f_{m+2}-f_{m-2}f_{m+1}^2\right).
	\end{aligned}
	\]
	\par
	Next we let $n \equiv 1 \Mod{4}$, so $n = 2m+1$ for $m$ even. Using Equation \ref{psiodd} gives
	\[
	\begin{aligned}
	f_2^2f_n &= \psi_2^4f_{2m+1} \\
	&= (\psi_2\psi_{m+2})(\psi_2\psi_m)^3 - f_2^2f_{m-1}f_{m+1}^3 \\
	&= f_{m+2}f_m^3 - f_2^2f_{m-1}f_{m+1}^3.
	\end{aligned}
	\]
	\par
	Finally, we treat $n \equiv 3 \mod{4}$. That is, $n=2m+1$ for $m$ odd. Again using Equation \ref{psiodd} yields
	\[
	f_{2m+1} = f_{m+2}f_m^3 - \psi_{m-1}\psi_{m+1}^3,
	\]
	so we have
	\[
	\begin{aligned}
	f_2^2f_{2m+1} &= \psi_2^4f_{2m+1} \\
	&= f_2^2f_{m+2}f_m^3 - (\psi_2\psi_{m-1})(\psi_2\psi_{m+1})^3 \\
	&= f_2^2f_{m+1}f_m^3 - f_{m-1}f_{m+1}^3.
	\end{aligned}
	\]
\end{proof}

\subsection{Proof of Proposition \ref{prop:2-adic valuation}} 
Let us now specialize to the case where $E$ is the elliptic curve defined by the Weierstrass equation $y^2 = x^3 + 2x^2 - 1$.
For reference we list the first four division polynomials for this particular elliptic curve.
\[
\begin{aligned}
f_1(x) &= 1, \\
f_2(x) &= 4x^3 + 8x^2 - 4, \\
f_3(x) &= 3x^4 + 8x^3 - 12x - 8, \\
f_4(x) &= 8x^9 + 48x^8 + 64x^7 - 168x^6 - 672x^5 - 896x^4 - 416x^3 + 192x^2 + 256x + 64.
\end{aligned}
\]
\begin{lem}\label{DivisibilityLemma}
	If $2 \mid n$, then $f_2 \mid f_n$.
\end{lem}
\begin{proof}
	Note that $f_2(x) = 4(x+1)(x^2+x-1)$ and that $x+1$ and $x^2+x-1$ have roots equal to the $x$-coordinates of the nontrivial $2$-torsion. Since $E[2] \subseteq E[n]$ whenever $2 \mid n$, we have that $(x+1)(x^2+x-1) \mid f_n$, so it suffices to show that $4 \mid f_n$. Let us write $n=2m$ and induct on $m$. The case of $m=1$ is trivial, and we use the recursive formula
	\[
	f_2f_{2m} = f_m\left(f_{m-1}^2f_{m+2}-f_{m-2}f_{m+1}^2\right)
	\]
	for the inductive step. Since $4 \mid f_2$, but $8 \nmid f_2$, it suffices to show that $16 \mid f_m\left(f_{m-1}^2f_{m+2}-f_{m-2}f_{m+1}^2\right)$. First suppose that $m$ is even, so that $m, m+2,$ and $m-2$ are all even. Then inductively, $4 \mid f_m, f_{m+2}, f_{m-2}$, which implies the result. Now if $m$ is odd, $m + 1$ and $m-1$ are even, so $16$ divides both $f_{m=1}^2$ and $f_{m+1}^2$.
\end{proof}
The main technical proposition is as follows.
\begin{prop}\label{FactList} Let $f_n$ be as above.
	\begin{enumerate}[label={(\arabic*)},ref={\theprop~(\arabic*)}]
		\item \label{FirstFact}
		\[
		\deg(f_n) = 
		\begin{cases}
		n^2/2 + 1 & n \textnormal{ even}, \\
		(n^2-1)/2 & n \textnormal{ odd}.
		\end{cases}
		\]
		
		\item  \label{SecondFact} The leading coefficient of $f_n(x)$ is $2n$ when $n$ is even and $n$ when $n$ is odd.
		\item \label{ThirdFact} If $n$ is odd, then $f_n(x) \equiv \pm x^{\frac{n^2-1}{2}} \Mod{4}$. 
		\item \label{FourthFact} Let $n$ be even and $k$ equal to the $2$-adic valuation of $n$. Then $2^{k+1} \mid f_n(x)$ in $\Z[x]$, and $\dfrac{1}{2^{k+1}}f_n(x) \equiv (x+1)(x^2+x+1)(x^{n^2/2-2}) \Mod{2}$. 
	\end{enumerate}
\end{prop}
Before embarking on the proof, let us point out that Proposition \ref{ThirdFact} and \ref{FourthFact} imply that, excepting the factors coming from the $1$ and $2$-torsion, the constant term of every irreducible factor of $f_n(x)$ has a factor of $2$, even after dividing out by the leading coefficient of $f_n(x)$. This is equivalent to the root of these irreducible factors having positive $2$-adic valuation. Proposition \ref{prop:2-adic valuation} and hence Theorem \ref{maintheorem2} then follow.
\begin{proof}[Proof of Proposition \ref{FactList}]
	We prove all parts by induction on $n$. The base cases are easy to check, so assume each statement is true for all indices less than or equal to $n-1$. \par
	We first prove Proposition \ref{FirstFact}. First suppose that $n$ is even and write $n=2m$. We then have
	\[
	f_2f_{2m} = f_m\left(f_{m-1}^2f_{m+2}-f_{m-2}f_{m+1}^2\right).
	\]
	Here we have to break into further cases where $m$ is even or odd. Let us treat $m$ even first. Then our inductive hypothesis implies that $\deg(f_m) = m^2/2+1$, $\deg(f_{m+2}) = (m+2)^2/2 + 1$, and $\deg(f_{m-2}) = (m-2)^2/2 + 1$. Moreover, $m \pm 1$ is odd, so $\deg(f_{m+1}^2) = m^2 + 2m$ and $\deg(f_{m-1}^2) = m^2 - 2m$. Hence, $\deg(f_{m-1}^2f_{m+2}) = \deg(f_{m-2}f_{m+1}^2)  = \dfrac{3m^2+6}{2}$. Our inductive hypothesis (in particular Proposition \ref{SecondFact}) implies that the leading coefficient of $f^2_{m-1}f_{m+2}$ is equal to $2(m+2)(m-1)^2 = 2m^3-6m+4$ whereas the leading coefficient of $f_{m-2}f_{m+1}^2$ is $2(m-2)(m+1)^2 = 2m^3-6m-4$, so coefficient of $x^{\frac{3m^2+6}{2}}$ in $f_{m-1}^2f_{m+2}-f_{m-2}f_{m+1}^2$ is $8$. In particular it's nonzero so that $\deg(f_{m-1}^2f_{m+2}-f_{m-2}f_{m+1}^2) = \dfrac{3m^2+6}{2}$. It follows that $\deg(f_2f_{2m}) = m^2/2+1 + \dfrac{3m^2+6}{2} = 2m^2+4$. Noting that $\deg(f_2) = 3$ gives that $\deg(f_{2m}) = \deg(f_n) = 2m^2+1 = n^2/2 + 1$. \par
	Now suppose that $m$ is odd. We still have the same recursive relation because $n$ is even. However since $m$ is odd, we now have $\deg(f_m) = (m^2-1)/2$, $\deg(f_{m-1}^2) = m^2-2m + 3$, $\deg(f_{m+2}) = (m^2+4m+3)/2$, so then $\deg(f_{m-1}^2f_{m+2}) = \frac{3}{2}(m^2+3)$. Similarly, $\deg(f_{m-2}) = (m^2-4m+3)/2$ and $\deg(f_{m+1}^2) = m^2 + 2m + 3$ together imply $\deg(f_{m-2}f_{m+1}^2) = \frac{3}{2}(m^2+3)$ as well. As before, one can compute using Proposition \ref{SecondFact} of the inductive hypothesis that the coefficient of $x^{\frac{3}{2}(m^2+3)}$ is nonzero (in fact equal to $16$), so $\deg(f_{m-1}^2f_{m+2}-f_{m-2}f_{m+1}^2) = \frac{3}{2}(m^2+3)$. Then we have $\deg(f_2f_{2m}) = 2m^2 + 4$, so $\deg(f_{2m}) = \deg(f_n) = 2m^2 + 1 = n^2/2 + 1$. \par
	To handle the case of $n$ odd, we have to separately consider when $n$ is congruent to $1$ or $3$ modulo $4$. If we write $n=2m+1$, these two cases are equivalent to $m$ even and odd, respectively. Let us treat $n \equiv 1 \Mod{4}$ first. The recursive relation is
	\[
	f_2^2f_{2m+1} = f_{m+2}f_m^3-f_2^2f_{m-1}f_{m+1}^3.
	\]
	Here $m$ is even, so $\deg(f_{m+2}) = (m^2+4m+6)/2$, and $\deg(f_m^3) = (3m^2+6)/2$. Combining these gives $\deg(f_{m+2}f_m^3) = 2(m^2+m+3)$. Similar calculations give $\deg(f_2^2f_{m-1}f_{m+1}^3) = 2(m^2+m+3)$. Again the leading coefficients do not cancel, so $f_{m+2}f_m^3 - f_2^2f_{m-1}f_{m+1}^3$ has degree $2(m^2+m+3)$ with leading coefficient $32m+16$. It follows that $\deg(f_{2m+1}) = 2(m^2+m)$, so $\deg(f_n) = (n^2-1)/2$. \par
	The last case is $n \equiv 3 \Mod{4}$. That is, $n=2m+1$ for $m$ odd. We have the relation
	\[
	f_2^2f_{2m+1} = f_2^2f_{m+2}f_m^3-f_{m-1}f_{m+1}^3.
	\] 
	In this case, $\deg(f_2^2) = 6$, $\deg(f_{m+2}) = \dfrac{(m+2)^2-1}{2}$, and $\deg(f_m^3) = 3\left(\dfrac{m^2-1}{2}\right)$. These imply $\deg(f_2^2f_{m+2}f_m^3) = 2(m^2+m+3)$. Similarly $\deg(f_{m-1}) = \dfrac{m^2-2m+3}{2}$ and $\deg(f_{m+1}^3) = \dfrac{3m^2+6m+9}{2}$ give that $\deg(f_{m-1}f_{m+1}^3) = 2(m^2+m+3)$. Similar to earlier cases, we then have that $f_2^2f_{m+2}f_m^3-f_{m-1}f_{m+1}^3$ is of degree $2(m^2+m+3)$ with leading coefficient $32m+16$. Then, as in the $n\equiv 1 \Mod{4}$ case, we have $\deg(f_n) = \deg(f_n) = (n^2-1)/2$, which completes the proof of Proposition \ref{FirstFact}. \par
	We now prove Proposition \ref{SecondFact}. For a polynomial $g$, let us write $\LC(g)$ for its leading coefficient. We have actually done the relevant calculations in the proof of Proposition \ref{FirstFact}. When $n \equiv 0 \Mod{4}$, so $n=2m$ with $m$ even, we have that $\LC\left(f_{m-1}^2f_{m+2}-f_{m-2}f_{m+1}^2\right) = 8$ so that $\LC(f_2f_{2m}) = 8\LC(f_m) = 16m$. Hence, $\LC(f_{2m}) = \LC(f_n) = 4m = 2n$. Similarly when $n=2m$ for $m$ odd, we have $\LC\left(f_{m-1}^2f_{m+2}-f_{m-2}f_{m+1}^2\right) = 16$, so $\LC(f_2f_{2m}) = 16\LC(f_m) = 16m$.
	For $n$ odd, we have that $\LC(f_2^2f_{2m+1}) = 16(2m+1)$ so that $\LC(f_{2m+1}) = 2m+1$. \par
	For Proposition \ref{ThirdFact}, we first treat the case where $n \equiv 1 \Mod{4}$, so we may write $n = 2m + 1$ for $m$ even. We note that the relation 
	\[
	f_2^2f_{2m+1} = f_{m+2}f_m^3-f_2^2f_{m-1}f_{m+1}^3
	\]
	implies that $f_2^2 \mid f_{m+2}f_m^3$. In fact, Lemma \ref{DivisibilityLemma} implies that $f_2$ divides both $f_{m+2}$ and $f_m$ so that $f_2^4 \mid f_{m+2}f_m^3$, which in particular implies that $\dfrac{f_{m+2}f_m}{f_2^2} \equiv 0 \Mod{4}$ as it has a factor of $f_2^2$, which is divisible by $4$. It follows that $f_{2m+1} \equiv -f_{m-1}f_{m+1}^3 \Mod{4}$. By the inductive hypothesis we have that $f_{m-1} \equiv \pm x^{(m^2-2m)/2} \Mod{4}$ and $f_{m+1} \equiv \pm x^{(m^2+2m)/2} \Mod{4}$, so $f_{2m+1} \equiv \pm x^{(2m^2+2m)} \Mod{4}$. That is, $f_n \equiv \pm x^{(n^2-1)/2} \Mod{4}$. The case of $n \equiv 3 \Mod{4}$ may be handled similarly, completing the proof of Proposition \ref{ThirdFact}. \par
	To begin the proof of Proposition \ref{FourthFact}, let us fix the notation that when $n$ is even, $f_n = f_2g_n$, and when $n$ is odd $f_n = g_n$. We break into cases based on the $2$-adic valuation of $n$. Let $k = v_2(n)$ be first equal to $1$, so we may write $n=2m$ for $m$ odd. In this case $m \pm 1$ are is even, so we have
	\[
	f_2f_{2m} = f_m(f_2^2g_{m-1}^2f_{m+2}-f_{m-2}f_2^2g_{m+1}^2),
	\]
	which implies
	\[
	f_{2m} = f_m(f_2g_{m-1}^2f_{m+2}-f_{m-2}f_2g_{m+1}^2).
	\]
	Now since $m$ is odd, exactly one of $m+1$ and $m-1$ is divisible by $4$. Say $m-1 \equiv 0 \Mod{4}$ (the other case is follows analogously). Then $f_{m-1}/4=f_2g_{m-1}/4$ is divisible by $2$ by inductive hypothesis. Then we have that
	
	\begin{align}
	\dfrac{f_{2m}}{4} &= f_m\left(\dfrac{f_2}{4}g_{m-1}^2f_{m+2}-f_{m-2}\dfrac{f_2}{4}g_{m+1}^2\right) \nonumber \\
	&\equiv (x+1)(x^2+x+1)f_mf_{m-2}g_{m+1}^2 \Mod{2}. \label{1/4f2m}
	\end{align}
	Now, since $g_{m+1}f_2=f_{m+1}$, the inductive hypothesis says 
	\[
	\dfrac{f_{m+1}}{4} \equiv (x+1)(x^2+x+1)x^{(m+1)^2/2-2} \Mod{2},
	\]
	so the fact that $f_2/4 \equiv (x+1)(x^2+x+1) \Mod{2}$ implies that $g_{m+1} \equiv x^{(m+1)^2/2-2} \Mod{2}$. So then since $f_m$ and $f_{m-2}$ both have odd indices, we may apply Proposition \ref{ThirdFact} to Equation \ref{1/4f2m} to obtain
	\[
	\begin{aligned}
	\dfrac{f_{2m}}{4} &\equiv (x+1)(x^2+x+1)x^{(m^2-1)/2}x^{((m-2)^2-1)/2}x^{(m+1)^2-4} \Mod{2} \\
	&\equiv (x+1)(x^2+x+1)x^{2m^2-2} \\
	&\equiv (x+1)(x^2+x+1)x^{n^2/2-2},
	\end{aligned}
	\]
	which handles the case of $k=1$. \par
	Next, suppose that $k=2$. We again write $n=2m$ and here have $m$ even with $v_2(m)=1$. Note that one of $m+2$ and $m-2$ will have $2$-adic valuation equal to $2$, and the other will have $2$-adic valuation at least $3$. Both possibilities lead to the same argument, so assume $v_2(m-2) \geq 3$ and hence $v_2(m+2) = 2$. Then we write $f_{m+2}=4(x+1)(x^2+x-1)g_{m-2}$. Then our inductive hypothesis implies
	\[
	\dfrac{f_{m+2}}{8} \equiv (x+1)(x^2+x+1)x^{(m+2)^2/2-2} \Mod{2}.
	\]
	So we have $g_{m+2}/2 \equiv x^{(m+2)^2/2-2} \Mod{2}$. The inductive hypothesis also implies that $16 \mid f_{m-2}$, and $f_{m-2} = 4(x+1)(x^2+x-1)g_{m-2}$, so $4 \mid g_{m-2}$. Hence, $g_{m-2}/2 \equiv 0 \Mod{2}$. Note that since $v_2(m)=1$ and $m-1$ is odd we have $f_m/4 \equiv (x+1)(x^2+x+1)x^{m^2/2-2} \Mod{2}$ and $f_{m-1}^2 \equiv x^{m^2-2m} \Mod{2}$. Combining all this we obtain
	\[
	\begin{aligned}
	\dfrac{f_{2m}}{8} &= \dfrac{f_m}{4}\left(f_{m-1}^2\dfrac{g_{m+2}}{2}-\dfrac{g_{m-2}}{2}f_{m+1}^2\right) \\
	&\equiv (x+1)(x^2+x+1)x^{2m^2-2} \\
	&\equiv (x+1)(x^2+x+1)x^{n^2/2-2}.
	\end{aligned}
	\]
	\par
	The last case to consider is $k \geq 3$. As always, write $n=2m$. Keeping the notation of the previous parts, we write our recursive relation as
	\[
	\dfrac{f_{2m}}{2^{k+1}} = \dfrac{f_2g_m}{2^k}\left(\dfrac{f_{m-1}^2g_{m+2}-g_{m-2}f_{m+1}^2}{2}\right).
	\]
	The leading coefficient of $f_{m-1}^2f_{m+2}-f_{m-2}f_{m+1}^2$ is $8$ by Proposition \ref{FirstFact} and Proposition \ref{SecondFact}. This implies that the leading coefficient of $f_{m-1}^2g_{m+2}-g_{m-2}f_{m+1}^2 =  (f_{m-1}^2f_{m+2}-f_{m+2}f_{m+1}^2)/f_2$ is $2$. Then,
	\begin{equation}
	\dfrac{f_{m-1}^2g_{m+2}-g_{m-2}f_{m+1}^2}{2} \label{fourthfactexpression1}
	\end{equation}
	is monic, and the degree can be checked to be correct using the inductive hypothesis. Then, it suffices to show that every monomial term in Equation \ref{fourthfactexpression1} other than the leading one is divisible by $2$. That is, we want to show that every monomial term besides the leading term in $f_{m-1}^2g_{m+2}-g_{m-2}f_{m+1}^2$ is divisible by $4$. Since $m \pm 1$ is odd, we have that every term other than the leading terms in $f_{m\pm 1}$ are divisible by $4$ by Proposition \ref{ThirdFact}, so we just need to show it for $g_{m \pm 2}$. Note that when $m$ is even,
	\[
	g_{2m} = g_m(f_{m-1}^2g_{m+2}-g_{m-2}f_{m+1}^2),
	\]
	so one may see that $g_{2m}$ only has its leading term not divisible by $4$ inductively.
\end{proof}
\begin{rem}[Affine Intersection Points]
Chu showed in \cite[Section 5.1]{ChuCharacterVarieties} that there are four affine points on the canonical component, $C$, for $7_4$ which intersect the other component of the character variety containing the character of an irreducible representation. They lie in a number field $L$ of degree $4$. These intersection points detect Seifert surfaces and so are also of geometric interest. The image of each of these points under the birational map $C \dashrightarrow E$ has $x$-coordinate equal to $1\pm i$; in particular they have positive a $2$-adic valuation. However they are still infinite order. One may compute with Magma (\cite{Magma}) or other software that the torsion subgroup of $E(L)$ is isomorphic to $\Z/6\Z$, but none of these four points has order less that or equal to $6$.  
\end{rem}
	\printbibliography
\end{document}